\documentclass[12pt, a4paper]{article}
\usepackage[utf8]{inputenc}
\usepackage[T1]{fontenc}
\usepackage{amsmath, amsthm, amsfonts, tikz, graphicx}
\usepackage[pdftex, colorlinks=true, linkcolor=blue, citecolor=blue, urlcolor=blue]{hyperref}

\newtheorem{theorem}{Theorem}[section]
\newtheorem{lemma}{Lemma}[section]
\newtheorem{corollary}{Corollary}[section]
\newtheorem{remark}{Remark}[section]
\newcommand{\x}{\boldsymbol{x}}

\title{Optimized Schwarz methods for heterogeneous heat transfer problems}

\author{Martin J. Gander\textsuperscript{1}, Liu-Di Lu\textsuperscript{1}, Tingting Wu\textsuperscript{2,3}}

\date{%
\small\textit{\textsuperscript{1}Section de Mathématiques, Université de Genève, rue du Conseil-Général 5-7, CP 64, 1205, Geneva, Switzerland}\\
\small\textit{\textsuperscript{2}School of Intelligent Equipment Engineering, Wuxi Taihu University, 214064, Wuxi, China}\\
\small\textit{\textsuperscript{3}State Key Laboratory of Mechanics and Control For Aerospace Structures, Nanjing University of Aeronautics and Astronautics, 210016, Nanjing, China}
}

\begin{document}

\maketitle

\begin{abstract}
We present here nonoverlapping optimized Schwarz methods applied to heat transfer problems with heterogeneous diffusion coefficients. After a Laplace transform in time, we derive the error equation and obtain the convergence factor. The optimal transmission operators are nonlocal, and thus inconvenient to use in practice. We introduce three versions of local approximations for the transmission parameter, and provide a detailed analysis at the continuous level in each case to identify the best local transmission conditions. Numerical experiments are presented to illustrate the performance of each local transmission condition. As shown in our analysis, local transmission conditions, which are scaled appropriately with respect to the heterogeneous diffusion coefficients, are more efficient and robust especially when the discontinuity of the diffusion coefficient is large.  
\end{abstract}


\section{Introduction}
Hypersonic vehicles often travel at speeds exceeding five times of the speed of sound, and due to this extreme speed, these vehicles are exposed to high aerodynamic and thermal loads~\cite{Anderson1989}. To ensure the safety of the vehicle, thermal protection structures must be designed and applied on the outer surface of the vehicle such that the inner structural temperature can stay in a sustainable range~\cite{Kumar2017}. Hence, it is vital to study the heat transfer problems in these critical areas to obtain the temperature of the vehicle. A typical illustration of thermal protection structures is shown in Figure~\ref{fig:illustration}. 
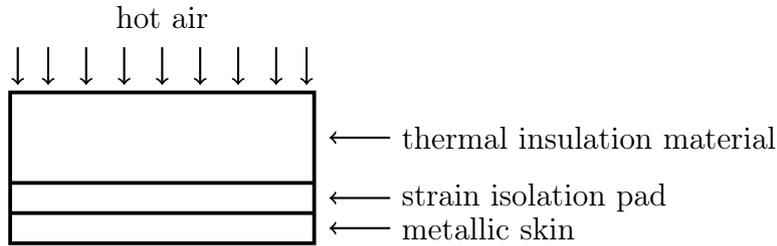
\begin{figure}
\centering
\begin{tikzpicture}
\draw[line width=0.5mm] (0,  0) rectangle (4,  2);
\draw[line width=0.5mm] (0,  0.4) -- (4,  0.4);
\draw[line width=0.5mm] (0,  0.8) -- (4,  0.8);
\draw[thick,  <-] (4.2,  0.2) -- (5,  0.2) node[right] {metallic skin};
\draw[thick,  <-] (4.2,  0.6) -- (5,  0.6) node[right] {strain isolation pad};
\draw[thick,  <-] (4.2,  1.4) -- (5,  1.4) node[right] {thermal insulation material};
\foreach \x in {0.1,  0.5,  1,  1.5,  2,  2.5,  3,  3.5,  3.9} {\draw[thick, ->] (\x,  2.6) -- (\x,  2.1);}
\node at (2,  3) {hot air};
\end{tikzpicture}
\caption{Illustration of thermal protection systems.}
\label{fig:illustration}
\end{figure}
Depending on the thermal protection techniques, several layers of materials can be applied over the vehicle skin, see e.g.~\cite{Uyanna2020} for a review. Each layer of the thermal protection structures may consist of different materials, such as aluminum and ceramic~\cite{Kumar2016}, and the diffusion coefficients can be very different from one material to another.

Numerical methods such as the finite element method and the boundary element method are often used to study such heat transfer problems, yielding reliable results~\cite{Zhang2019,Feng2016}. However, simulating heat transfer across various materials for critical areas of the vehicle can be time consuming. In~\cite{Girault2005,Girault20052}, the Reduced Models method is used to solve a nonlinear heat conduction problem, which drastically reduces the computing time. Given the geometric structure presented in Figure~\ref{fig:illustration}, nonoverlapping domain decomposition methods are natural candidates to introduce parallelism and accelerate the numerical solution of heat transfer problems with heterogenous diffusion coefficients. In~\cite{Divo2003}, the authors developed a domain decomposition, or artificial subsectioning technique, along with a boundary--element method, to solve such heat conduction problems, showing the potential of  domain decomposition.

The idea of domain decomposition was initially introduced by Hermann Amandus Schwarz in~\cite{Schwarz1870} to prove rigorously the existence of solution for Laplace problems. His method has then been developed as a computational tool with the arrival of parallel computing, see e.g.~\cite{Gander2008} for a historical review. Unlike dealing with homogeneous heat transfer problems where a continuous diffusion function is considered over the entire domain, the heterogeneity of the material between two subdomains require special attention for heterogeneous heat transfer problems. In~\cite{Maday2007, Dubois2015}, optimized Schwarz methods are analyzed for solving heterogeneous Laplace problems. A reaction--diffusion problem with heterogenous coefficients is studied in~\cite{Gander2019}. In~\cite{Gerardo-Giorda2005}, the authors consider using optimized Schwarz methods for solving unsymmetric advection--diffusion--reaction problems with strongly heterogenous and anisotropic diffusion coefficients. The balancing Neumann--Neumann method is applied in~\cite{Goldfeld2003} to treat linear elasticity systems with discontinuous coefficients. In~\cite{Gander2016}, the authors extend the study to parabolic heat transfer problems with a constant diffusion coefficient using Dirichlet--Neumann and Neumann--Neumann waveform relaxation methods. Optimized Schwarz waveform relaxation methods are considered in~\cite{Lemarie2013part1, Lemarie2013part2} to solve heterogeneous heat transfer problems. More recently, the authors in~\cite{Birken2024} analyzed at the continuous level of the Dirichlet--Neumann waveform relaxation method applied to heterogeneous heat transfer problems.

In the current study, we focus on the optimized Schwarz waveform relaxation methods to solve heat transfer problems with heterogeneous diffusion coefficients. It has already been observed in~\cite{Lemarie2013part1, Lemarie2013part2} that the optimal transmission operators are nonlocal in time, and thus are inconvenient to use in practise. For this reason, we introduce here three local approximations of the transmission operators by taking into account the heterogenous diffusion coefficients. As these local approximations are scaled differently with respect to the diffusion coefficients, we analyze in detail the min-max problem associated with each approximation and find analytical formulas for the optimized local transmission parameters. In particular, we show that the equioscillation property does not always lead to the best transmission parameters, as reported also in~\cite{Gander2016}. Thus, one needs to be careful when addressing the min-max problems to characterize the best transmission parameters. In addition, we also show the importance of using a good scaling to be able to derive an efficient and robust solver in the case of a largely heterogeneous media.

Our paper is organized as follows: in Section~\ref{sec:2}, we introduce the heterogeneous heat transfer problem and optimized Schwarz methods. A Laplace analysis is applied to the error equations to determine the convergence factor. In Section~\ref{sec:3}, we introduce three local approximations of the optimal transmission operators and provide a detailed analysis of each associated min-max problem. Numerical experiments are presented in Section~\ref{sec:4} to illustrate the performance of these local transmission conditions.

\section{Model problem}\label{sec:2}
To model the heat transfer between different materials as shown in Figure~\ref{fig:illustration}, we consider the heterogeneous heat equation
\begin{equation}\label{eq:heat}
\begin{aligned}
\partial_t u &= \nabla \cdot (\nu \nabla u)  + f  &&\text{ in } Q:=\Omega \times(0,  T), 
\\
u &= u_0 &&\text{ on } \Sigma_0:=\Omega\times\{0\}, 
\\
u &= g &&\text{ on } \Sigma:=\partial\Omega\times(0,  T), 
\end{aligned}
\end{equation}
where $\Omega\subset\mathbb{R}^d$,  $d=1,  2,  3$, with its boundary $\partial\Omega$, $T$ is the fixed final time, $\nu$ is the heat diffusion function, $f$ is the source term, $u_0$ is the initial condition, and $g$ represents some Dirichlet boundary conditions. Furthermore, we consider a natural decomposition of two nonoverlapping subdomains $\Omega_1$ and $\Omega_2$ such that $\Omega_1\cap\Omega_{2} = \Gamma$, with $\Gamma$ the interface between $\Omega_1$ and $\Omega_2$, as shown in Figure~\ref{fig:HeatJump}. The heat diffusion function $\nu$ is assumed to be a piecewise constant function in space, where $\nu (\x) =  \nu_j$ for $\x \in\Omega_j$ with $\nu_j>0$, $j=1, 2$. For the sake of brevity, we will omit the initial and boundary conditions in the following.
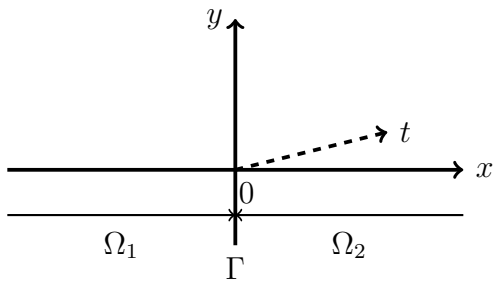
\begin{figure}
\centering
\begin{tikzpicture}
\draw[line width=0.5mm, ->] (-3,  0)  -- (3,  0)  node[right] {$x$};
\draw[line width=0.5mm, ->] (0, -1) -- (0, 2) node[left] {$y$};
\draw[line width=0.5mm, dashed, ->] (0, 0) -- (2, 0.5) node[right] {$t$};
\node at (0.15, -0.3) {0};
\draw[thick, ->] (-3, -0.6) -- (0, -0.6);
\node at (-1.5, -1) {$\Omega_1$};
\draw[thick, ->] (3, -0.6) -- (0, -0.6);
\node at (1.5, -1) {$\Omega_2$};
\node at (0, -1.3) {$\Gamma$};
\end{tikzpicture}
\caption{2D illustration of the decomposition.}
\label{fig:HeatJump}
\end{figure}

The following physical coupling conditions are applied on the interface 
\[u_1 = u_2, \quad  \nu_1 \partial_{\mathbf{n}_1}u_1= -\nu_2 \partial_{\mathbf{n}_2}u_2, \quad \text{ on } \Sigma:=\Gamma\times(0, T),\]
to ensure the continuity of the solution and its normal flux between the subdomains. Here, the unit outward normal vector is denoted by $\mathbf{n}_j$. According to these two physical coupling conditions, we can write the optimized Schwarz method as: for the iteration index $k = 1, 2, \ldots$, one solves
\begin{equation}\label{eq:OS}
\begin{aligned}
\partial_t u_1^{k+1}  &= \nu_1 \Delta u_1^{k+1} + f_1  &&\text{ in }  Q_1,
\\
\big(\nu_1 \partial_{\mathbf{n}_1}+S_1\big) u_1^{k+1} &= \big(\nu_{2} \partial_{\mathbf{n}_1}+S_{1}\big) u_{2}^k  && \text{ on } \Sigma, 
\\
\partial_t u_{2}^{k+1}  &= \nu_{2} \Delta u_{2}^{k+1} + f_2  &&\text{ in }  Q_2,
\\
\big(\nu_{2} \partial_{\mathbf{n}_2}-S_2\big) u_{2}^{k+1} &= \big(\nu_{1} \partial_{\mathbf{n}_2}-S_2\big) u_{1}^{k+1}  && \text{ on } \Sigma, 
\end{aligned}
\end{equation}
with $Q_j:=\Omega_j\times(0, T)$, $j=1, 2$. The system~\eqref{eq:OS} is then completed by the given initial and boundary conditions of the problem~\eqref{eq:heat}. Here, $f_j$ denotes the source term $f$ restricted to the space-time domain $Q_j$, and $S_j$ is a linear space-time operator. As illustrated in Figure~\ref{fig:HeatJump}, the decomposition is only in the $x$-direction, we thus consider in the following the one dimensional case, i.e., $\Omega=\mathbb{R}$, to focus on the transmission condition at the interface $x=0$. This will simplify the computations and allow us to obtain a more compact analytical form. In this case, the two space-time subdomains are $Q_1=(-\infty, 0)\times(0, T)$ and $Q_2=(0, \infty)\times(0, T)$, and the linear operator $S_j$ is only related to the time variable. Although the following convergence analyses are for the two-subdomain case only, our numerical experiments in Section 4 for multiple subdomains with different choices of the diffusion coefficient $\nu$ show that our theoretical results are also very useful in more general situations.

\subsection{Laplace Analysis} 
To understand the convergence behavior of the optimized Schwarz algorithm~\eqref{eq:OS}, we will study the associated error equations with solutions which go to zero when $x$ goes to infinity. We denote the error by $e_j^k(\x, t):=u(\x, t)-u_j^k(\x, t)$, $j=1, 2$, which satisfies by linearity the equation
\[\partial_t e_j^{k} = \partial_t \big(u -u_j^k\big)
=\nu_j \Delta \big(u- u_j^k\big)
=\nu_j\Delta e_j^{k} \quad \text{ in } Q_j.\]
To focus on the transmission condition in space at the interface $\Gamma$, we apply a Laplace transform in the time variable $t$, 
\[\hat{e}_j^k(\x,  s) := \mathcal{L}\{e_j^k(\x,  t)\} 
= \int _{0}^{\infty} e_j^k(\x,  t) e^{-st} \,  \mathrm{d}t,\]
where $s\in\mathbb{C}$ is a complex number. We study the associated error equation of~\eqref{eq:OS} after the Laplace transform, that is,
\begin{equation}\label{eq:OSLaplace}
\begin{aligned}
s \hat{e}_1^{k+1}(x,  s) &= \nu_1 \partial_{xx} \hat{e}_1^{k+1}(x,  s)  && \text{in} \ Q_1, \\
\big(\nu_1 \partial_x + \sigma_1(s)\big) \hat{e}_1^{k+1}(0,  s) &= \big(\nu_2 \partial_x + \sigma_1(s)\big) \hat{e}_2^{k}(0,  s),  \\
s \hat{e}_2^{k+1}(x,  s) &= \nu_2 \partial_{xx} \hat{e}_2^{k+1}(x,  s)  && \text{in} \ Q_2, \\
\big(\nu_2 \partial_x - \sigma_2(s)\big) \hat{e}_2^{k+1}(0,  s) &=\big(\nu_1 \partial_x -\sigma_2(s)\big)  \hat{e}_1^{k+1}(0,  s),
\end{aligned}
\end{equation} 
where $\sigma_j(s)$ are the Laplace symbols of the operators $S_j$. The general solutions are given by
\[\hat{e}_1^{k+1}(x,  s) = C_1^{k+1}(s) \hat{e}^{\frac{\sqrt{s}}{\sqrt{\nu_1}} x}, 
\quad 
\hat{e}_2^{k+1}(x,  s) = C_2^{k+1}(s) \hat{e}^{-\frac{\sqrt{s}}{\sqrt{\nu_2}} x}.\]
Applying the transmission conditions in~\eqref{eq:OSLaplace},  we obtain the convergence factor for $\{\hat{e}_j^{k}\}_{k=1,  2,  \ldots}$
\begin{equation}\label{eq:rhoehat}
\rho(s, \sigma_1,  \sigma_2) := 
\left| 
\frac{\sigma_1(s)-\sqrt{\nu_2} \sqrt{s}} 
{\sigma_1(s)+\sqrt{\nu_1} \sqrt{s}}  
\cdot  
\frac{\sigma_2(s)-\sqrt{\nu_1} \sqrt{s}} 
{\sigma_2(s)+\sqrt{\nu_2} \sqrt{s}} 
\right|.  	
\end{equation}
It is straightforward from~\eqref{eq:rhoehat} that we can get optimal convergence by choosing
\begin{equation}\label{eq:sigopt}
\sigma_1(s) = \sqrt{\nu_2} \sqrt{s},  
\quad 
\sigma_2(s) = \sqrt{\nu_1} \sqrt{s}.
\end{equation}
This leads to convergence in two iterations,  since the errors at iteration $k=2$ vanish. However, the best choice is nonlocal in time due to the term $\sqrt{s}$, and it is expensive to compute and inconvenient for the implementation. Therefore, the goal of the current study is to find good local approximations of $\sigma_j(s)$ that can still give fast convergence.

\section{Approximation of the optimal operators}\label{sec:3}
The idea is to fix a class of possible transmission conditions $\mathcal{C}$ and uniformly optimize the convergence factor over a range of frequencies for our problem. This corresponds to solve the min-max problem 
\begin{equation}\label{eq:minmax}
\min_{\sigma_j \in \mathcal{C}} \left(\max_{ s } \rho (s, \sigma_1, \sigma_2) \right).
\end{equation}
To find local approximations of $\sigma_j(s)$, we consider in the following $\sigma_j \in \mathbb{R}$, independent of the time variable. In this way, the convergence factor~\eqref{eq:rhoehat} becomes
\begin{equation}\label{eq:rhoehatR}
\rho(s, \sigma_1,  \sigma_2) := 
\left| 
\frac{\sigma_1-\sqrt{\nu_2} \sqrt{s}}
{\sigma_1+\sqrt{\nu_1} \sqrt{s}}  
\cdot  
\frac {\sigma_2 -\sqrt{\nu_1} \sqrt{s}}
{\sigma_2 + \sqrt{\nu_2} \sqrt{s}} 
\right|.  
\end{equation}
For the Laplace transform, we have $s=\eta +i\omega$ with $\eta,  \omega \in \mathbb{R}$. This implies that
\[\sqrt {s}=\sqrt{\eta +i\omega }
=\sqrt{\frac{\eta+\sqrt{\eta^2+\omega^2}}{2}} \pm i \sqrt{\frac{-\eta+\sqrt{\eta^2+\omega^2}}{2}}.\]
Since $\sqrt{s}$ is an even function of the imaginary part $\omega$, the convergence factor $\rho$ is also an even function of $\omega$. Therefore, we only consider $\omega \geq 0$ in the analysis. Now the imaginary part $\omega=0$ corresponds to a constant function in time, and since the error function $e_j^k(x, t)$ equals zero at $t=0$,  the constant function cannot be part of the error function in the iteration. From a numerical viewpoint, when solving the problem in the time interval $[0, T]$, we can heuristically state that $\omega\in[\omega_{\min}, \omega_{\max}]$, where the smallest frequency $\omega_{\min}$ is $\frac{\pi}{2T}$, and the largest frequency is related to the time step $\Delta t$, that is $\omega_{\max}=\frac{\pi}{\Delta t}$. We refer to~\cite[Figure 3.17]{Gander2024} for more details about this statement. Thus, we can set $\eta=0$ as we only solve the min-max problem~\eqref{eq:minmax} away from $\omega=0$. Denoting by $\widetilde{\omega}:=\sqrt{\frac{\omega}{2}}$, we get
\[\sqrt {s}=\sqrt{\frac{\omega}{2}} \pm i\sqrt{\frac{\omega}{2}} 
=\widetilde{\omega} \pm i\widetilde{\omega}. \]
The new parameter $\widetilde{\omega} \in [\widetilde{\omega}_1, \widetilde{\omega}_2]$ with $\widetilde{\omega}_1:=\sqrt{\frac{\omega_{\min}}{2}}=\sqrt{\frac{\pi}{4T}}$ and $\widetilde{\omega}_2:=\sqrt{\frac{\omega_{\max}}{2}}=\sqrt{\frac{\pi}{2\Delta t}}$. The convergence factor~\eqref{eq:rhoehatR} can then be simplified to 
\begin{equation}\label{eq:rhoehatRapprox}
\rho(\widetilde{\omega}, \sigma_1, \sigma_2) = 
\sqrt{ 
\frac{(\sigma_1-\sqrt{\nu_2} \widetilde{\omega})^2+\nu_2 \widetilde{\omega}^2 }
{(\sigma_1+\sqrt{\nu_1}\widetilde{\omega})^2+\nu_1\widetilde{\omega}^2} 
\cdot 
\frac{(\sigma_2-\sqrt{\nu_1} \widetilde{\omega})^2+\nu_1 \widetilde{\omega}^2 }
{(\sigma_2+\sqrt{\nu_2}\widetilde{\omega})^2+\nu_2\widetilde{\omega}^2} 
}.
\end{equation}
To find good local operators, we can restrict the range of $\sigma_j$. More precisely, suppose $ \sigma_1>0$ and substitute $\sigma_1$ by $-\sigma_1$ in~\eqref{eq:rhoehatRapprox}, we have
\[\rho(\widetilde{\omega}, -\sigma_1, \sigma_2) = 
\sqrt{
\frac{(\sigma_1 + \sqrt{\nu_2} \widetilde{\omega})^2 + \nu_2 \widetilde{\omega}^2 }
{(\sigma_1 - \nu_1 \widetilde{\omega}^2) + \nu_1\widetilde{\omega}^2}  
\cdot  
\frac{(\sigma_2 - \sqrt{\nu_1} \widetilde{\omega})^2 + \nu_1 \widetilde{\omega}^2} 
{(\sigma_2 + \sqrt{\nu_2}\widetilde{\omega})^2 + \nu_2 \widetilde{\omega}^2} }.\]
This implies that $\rho(\widetilde{\omega}, -\sigma_1, \sigma_2) >\rho(\widetilde{\omega}, \sigma_1, \sigma_2)$, when $ \sigma_1>0 $. Therefore, for fast convergence, $\sigma_1>0$ should be chosen. In a similar way, we can restrict the range of $\sigma_2$ to $\sigma_2 >0$. The min-max problem~\eqref{eq:minmax} thus becomes
\begin{equation}\label{eq:P}
\min_{\sigma_j >0} \left(\max_{\widetilde{\omega}_1 \leq \widetilde{\omega} \leq \widetilde{\omega}_2} \rho(\widetilde{\omega}, \sigma_1, \sigma_2) \right). \tag{P}
\end{equation}
Before analyzing the convergence of several choices for local transmission parameters $\sigma_j$, we give sufficient conditions on $\sigma_j$ that will guarantee convergence of the optimized Schwarz algorithm~\eqref{eq:OSLaplace}.

\begin{theorem}[Sufficient condition]\label{thm:SuffCond} 
Under the conditions 
\[\begin{aligned}
&0 < \sigma_2 \leq \sigma_1, &&\text{if} \ \nu_1 < \nu_2, 
\\
&0 < \sigma_1 \leq \sigma_2, &&\text{if} \ \nu_2 < \nu_1, 
\end{aligned}\]
the optimized Schwarz algorithm~\eqref{eq:OSLaplace} converges for all $ \widetilde{\omega}\in [\widetilde{\omega}_1,   \widetilde{\omega}_2]$ and the convergence factor~\eqref{eq:rhoehatRapprox} satisfies
\[\rho (\widetilde{\omega}, \sigma_1, \sigma_2)<1.\]
\end{theorem}

\begin{proof}
To guarantee convergence of the optimized Schwarz algorithm~\eqref{eq:OSLaplace}, we want from~\eqref{eq:rhoehatRapprox} that
\[\rho(\widetilde{\omega}, \sigma_1, \sigma_2) = 
\sqrt{ 
\frac{(\sigma_1-\sqrt{\nu_2} \widetilde{\omega})^2+\nu_2 \widetilde{\omega}^2 } 
{(\sigma_1+\sqrt{\nu_1} \widetilde{\omega})^2+\nu_1 \widetilde{\omega}^2} 
\cdot 
\frac{(\sigma_2-\sqrt{\nu_1} \widetilde{\omega})^2+\nu_1 \widetilde{\omega}^2} 
{(\sigma_2+\sqrt{\nu_2} \widetilde{\omega})^2+\nu_2 \widetilde{\omega}^2} 
} <1,\]
which can be simplified to
\[\widetilde{\omega}(\sqrt{\nu_1} - \sqrt{\nu_2})(\sigma_1 - \sigma_2)
-\sigma_1 \sigma_2
-2 \sqrt{\nu_1}\sqrt{\nu_2}\widetilde{\omega}^2 < 0. \]
A simple sufficient condition for this inequality to hold is $(\sqrt{\nu_1}-\sqrt{\nu_2})(\sigma_1-\sigma_2) \leq 0$, which is clearly not a necessary condition. This concludes the proof.
\end{proof}

In the following subsections, we consider three choices for the transmission parameters $\sigma_j$ and their related min-max problems~\eqref{eq:P}. In all cases, Theorem~\ref{thm:SuffCond} will be satisfied to guarantee convergence of optimized Schwarz algorithm~\eqref{eq:OSLaplace} when using these local transmission conditions. To treat the min-max problems~\eqref{eq:P} and find the best transmission parameters $\sigma_j$, we follow three steps similar as used in~\cite{Dubois2015}:
\begin{enumerate}
\item restrict the range of the transmission parameter $\sigma_j$ with respect to the frequencies $\widetilde{\omega}_1$ and $\widetilde{\omega}_2$;
\item identify possible local maximum points $\widetilde{\omega}$ for the min-max problem~\eqref{eq:P};
\item analyze how these local maxima behave when the transmission parameters $\sigma_j$ vary to find the minimizers.
\end{enumerate}

\subsection{Local transmission parameter: Version I}\label{sec:v1}
We first consider the transmission parameters $\sigma_j$ with one free variable $p$,
\begin{equation}\label{eq:sgmv1}
\sigma_1 =\sigma_2=\sqrt{\nu_2} p,  \quad p > 0,
\end{equation}
where we scale both parameters with only one diffusion coefficient $\nu_2$. Note that one could also scale with respect to $\nu_1$ instead. Here, the parameter $p$ is chosen to be positive such that the hypothesis in Theorem~\ref{thm:SuffCond} is satisfied, and thus the convergence of~\eqref{eq:OSLaplace} is guaranteed. Although this choice may not be the best one, as the optimal transmission operators~\eqref{eq:sigopt} are scaled with respect to both diffusion coefficients $\nu_1$ and $\nu_2$, we still analyze this very simple choice both for completeness and comparison purposes. The convergence factor~\eqref{eq:rhoehatRapprox} for this choice is given by
\begin{equation}\label{eq:rhov1}
\rho(\widetilde{\omega}, \, p) = 
\sqrt{ 
\frac{(p - \widetilde{\omega})^2 + \widetilde{\omega}^2 } 
{(p + \mu \widetilde{\omega})^2 + \mu^2 \widetilde{\omega}^2}  
\cdot 
\frac{(p - \mu \widetilde{\omega})^2 + \mu^2 \widetilde{\omega}^2}
{(p + \widetilde{\omega})^2 + \widetilde{\omega}^2} 
}, 
\end{equation}
where $\mu := \sqrt{\frac{\nu_1}{\nu_2}}$ such that $\mu^2$ is the ratio of the two diffusion coefficients. In the following, we only consider the case when $\mu>1$, since the case when $\mu<1$ can be converted to the case $\mu>1$ by interchanging $\nu_1$ and $\nu_2$. We now want to find the best value of the transmission parameter $p$ such that the convergence factor~\eqref{eq:rhov1} can be minimized uniformly over the range of frequencies $[\widetilde{\omega}_1, \widetilde{\omega}_2]$. In this way, the min-max problem~\eqref{eq:P} becomes
\begin{equation}\label{eq:P1}
\min_{p > 0} \left(\max_{\widetilde{\omega}_1 \leq \widetilde{\omega} \leq \widetilde{\omega}_2} \rho (\widetilde{\omega}, p) \right).  
\tag{P1}
\end{equation}
We first show how to restrict the range for the transmission parameter $p$.

\begin{lemma}[Restrict parameter $p$]\label{lem:restrictpv1}
The min-max problem~\eqref{eq:P1} is equivalent to the problem where we minimize the convergence factor when the transmission parameter $p$ is in the interval 
\[p\in\left\{
\begin{aligned}
&\big[\sqrt{2\mu }\widetilde{\omega}_1, \; \sqrt{2\mu }\widetilde{\omega}_2\big],  
&& \text{if} \ \mu \leq 2+\sqrt{3}, \\
&\Big[\widetilde{\omega}_1\sqrt{(\mu-1)^2-\delta}, \; \widetilde{\omega}_2\sqrt{(\mu-1)^2+\delta}\Big],  
&& \text{if} \ \mu > 2+\sqrt{3}, 
\end{aligned}
\right.\]
with $\delta =\sqrt{(\mu^2-4\mu+1)(\mu^2+1)}$.
\end{lemma}

\begin{proof}
We first take the partial derivative of the convergence factor~\eqref{eq:rhov1} with respect to the transmission parameter $p$,
\begin{equation}\label{eq:signDrhoDp}
\text{sign}\Big(\frac{\partial{\rho}}{\partial{p}}\Big) = 
\text{sign}\Big(
(p^2 - 2\mu \widetilde{\omega}^2)
\big(p^4 - 2 p^2(\mu - 1)^2 \widetilde{\omega}^2 + 4\mu^2 \widetilde{\omega}^4\big)
\Big).
\end{equation}
The discriminant of the second polynomial $p^4- 2p^2(\mu - 1)^2 \widetilde{\omega}^2 + 4\mu^2 \widetilde{\omega}^4$ is
\begin{equation}\label{eq:Delta}
\Delta=4\widetilde{\omega}^4 (\mu^2-4\mu+1)(\mu^2+1)  .
\end{equation}
According to the value of the discriminant~\eqref{eq:Delta}, we divide the analysis into two cases.

\textbf{Case 1 $\Delta \leq 0$:} In this case, we find from~\eqref{eq:Delta} that $\mu \leq 2+\sqrt{3}$, and the polynomial $p^4- 2p^2(\mu - 1)^2 \widetilde{\omega}^2 + 4\mu^2 \widetilde{\omega}^4$ is always nonnegative. Thus, we have
\[\text{sign}\Big(\frac{\partial{\rho}}{\partial{p}}\Big)= 
\text{sign}\big(p^2-2\mu \omega^2\big)
=\left\{\begin{aligned}
&\text{positive}, && \text{if}  \ p>\sqrt{2\mu }\widetilde{\omega}, \\
&\text{negative}, && \text{if} \ p<\sqrt{2\mu} \widetilde{\omega}.
\end{aligned}
\right.\]
We observe that increasing $p$ will make the convergence factor~\eqref{eq:rhov1} decrease when $p<\sqrt{2\mu} \widetilde{\omega}_1$,  and decreasing $p$ will make the convergence factor~\eqref{eq:rhov1} decrease when $p>\sqrt{2\mu} \widetilde{\omega}_2$.  Therefore,  $p$ should be in the range of $[\sqrt{2\mu }\widetilde{\omega}_1, \sqrt{2\mu }\widetilde{\omega}_2]$ to minimize the convergence factor $\rho$.

\textbf{Case 2 $\Delta >0$:} In this case, we find from~\eqref{eq:Delta} that $\mu > 2+\sqrt{3}$. From~\eqref{eq:signDrhoDp}, we then find 
\[\text{sign}\Big(\frac{\partial{\rho}}{\partial{p}}\Big)= 
\left\{\begin{aligned}
&\text{negative}, && \text{if}  \ 0 < p^2 < \widetilde{\omega}^2\big((\mu-1)^2- \delta\big), \\
&\text{positive}, && \text{if}  \ \widetilde{\omega}^2\big((\mu-1)^2- \delta\big) < p^2 < 2\mu \widetilde{\omega}^2, \\
&\text{negative}, && \text{if} \ 2\mu \widetilde{\omega}^2 < p^2 < \widetilde{\omega}^2\big((\mu-1)^2+\delta\big),\\
&\text{positive}, && \text{if} \ p^2 > \widetilde{\omega}^2\big((\mu-1)^2+\delta\big).
\end{aligned}
\right.\]
Similar to \textbf{Case 1},  $p^2$ should be in the range of $[\widetilde{\omega}_1^2((\mu-1)^2-\delta), \, \widetilde{\omega}_2^2((\mu-1)^2+\delta)]$ to minimize the convergence factor $\rho$. This completes the proof. 
\end{proof}

We now study the behavior of the convergence factor~\eqref{eq:rhov1} as a function of $\widetilde{\omega}$.

\begin{lemma}[Local maxima of $\widetilde{\omega}$]\label{lem:maxomegav1}
Denoting by $\widetilde{\omega}_c := \frac{p}{\sqrt{2 \mu}}$, we can write the maximum of the convergence factor~\eqref{eq:rhov1} as
\[\begin{aligned}
&\text{if} \ 
\mu \leq 2+\sqrt{3},  
&&\max_{\widetilde{\omega}_1 \leq \widetilde{\omega} \leq \widetilde{\omega}_2} \rho(\widetilde{\omega}, p) 
=\max\left\{
\rho (\widetilde{\omega}_1, p), \,
\rho (\widetilde{\omega}_2, p)
\right\},  
\\
&\text{if} \ 
\mu > 2+\sqrt{3},  
&&\max_{\widetilde{\omega}_1 \leq \widetilde{\omega} \leq \widetilde{\omega}_2} \rho(\widetilde{\omega}, p) 
= \left\{\begin{aligned}
&\max\left\{\rho (\widetilde{\omega}_1, p), \,
\rho (\widetilde{\omega}_2, p)\right\},  
\widetilde{\omega}_c \notin[\widetilde{\omega}_1, \widetilde{\omega}_2],\\
&\max\left\{\rho (\widetilde{\omega}_1, p), \,
\rho (\widetilde{\omega}_c, p), \,
\rho (\widetilde{\omega}_2, p)\right\},   
\widetilde{\omega}_c \in [\widetilde{\omega}_1, \widetilde{\omega}_2].	 
\end{aligned}\right.    		
\end{aligned}\]
\end{lemma}

\begin{proof}
Taking the partial derivative of the convergence factor~\eqref{eq:rhov1} with respect to the frequency $\widetilde{\omega}$, we find
\[\text{sign} \Big(\frac{\partial{\rho}}{\partial{\widetilde{\omega}}}\Big) = 
\text{sign} \Big(-(p^2 - 2\mu \widetilde{\omega}^2)
(p^4 - 2p^2(\mu - 1)^2 \widetilde{\omega}^2 + 4\mu^2 \widetilde{\omega}^4)\Big),\]
which has the opposite sign of~\eqref{eq:signDrhoDp}. Given this similarity between the two partial derivatives, we also consider two cases.

\textbf{Case 1 $\mu \leq 2+\sqrt{3}$:} In this case, the discriminant~\eqref{eq:Delta} is non-positive, and the polynomial $p^4- 2p^2(\mu - 1)^2 \widetilde{\omega}^2 + 4\mu^2 \widetilde{\omega}^4$ is always nonnegative. Then, we have
\[\text{sign}\Big(\frac{\partial{\rho}}{\partial{\widetilde{\omega}}}\Big)= 
\text{sign}\big(2\mu \omega^2 - p^2\big)
=\left\{\begin{aligned}
&\text{negative}, && \text{if}  \ \widetilde{\omega}_1 < \widetilde{\omega} < \widetilde{\omega}_c, \\
&\text{positive}, && \text{if} \ \widetilde{\omega}_c < \widetilde{\omega} < \widetilde{\omega}_2,
\end{aligned}
\right.\]
meaning that the maximum of the convergence factor $\rho(\widetilde{\omega}, p)$ in the range $[\widetilde{\omega}_1, \widetilde{\omega}_2]$ is $\max\{\rho(\widetilde{\omega}_1, p), \,\rho(\widetilde{\omega}_2, p)\}$.

\textbf{Case 2 $\mu > 2+\sqrt{3}$:} In this case, we observe that,  
\[\text{sign}\Big(\frac{\partial{\rho}}{\partial{\widetilde{\omega}}}\Big)= 
\left\{\begin{aligned}
&\text{negative}, && \text{if}  \ 0 < \widetilde{\omega}^2 < \frac{\widetilde{\omega}_c^2}{2\mu}\big((\mu-1)^2- \delta\big), \\
&\text{positive}, && \text{if}  \ \frac{\widetilde{\omega}_c^2}{2\mu}\big((\mu-1)^2- \delta\big) < \widetilde{\omega}^2 < \widetilde{\omega}_c^2, \\
&\text{negative}, && \text{if} \ \widetilde{\omega}_c^2 < \widetilde{\omega}^2 < \frac{\widetilde{\omega}_c^2}{2\mu}\big((\mu-1)^2+\delta\big),\\
&\text{positive}, && \text{if} \ \widetilde{\omega}^2 > \frac{\widetilde{\omega}_c^2}{2\mu}\big((\mu-1)^2+\delta\big).
\end{aligned}
\right.\]
As the value of $\widetilde{\omega}_c=\frac{p}{\sqrt{2 \mu}}$ might fall outside the interval $[\widetilde{\omega}_1, \, \widetilde{\omega}_2]$, the maximum of the convergence factor $\rho(\widetilde{\omega}, p)$ will then be taken according to the value of $\widetilde{\omega}_c$. This concludes the proof.
\end{proof}

With the help of Lemma~\ref{lem:restrictpv1} and Lemma~\ref{lem:maxomegav1}, we can now identify the possible choices of the optimized parameter $p$ according to the ratio $\mu$.

\begin{theorem}[Optimized transmission parameter: $\mu \leq 2+\sqrt{3}$]\label{thm:optparav1}
The value $p$ minimizing the convergence factor~\eqref{eq:rhov1} is $p^* = \sqrt{2 \mu \widetilde{\omega}_1 \widetilde{\omega}_2}$.
\end{theorem}

\begin{proof}
In this case, the maximum in the min-max problem~\eqref{eq:P1} is determined by Lemma~\ref{lem:maxomegav1} as $\max\left\{\rho(\widetilde{\omega}_1, p), \, \rho(\widetilde{\omega}_2, p)\right\}$, and we need to find its minimum with respect to $p$. According to~\eqref{eq:signDrhoDp}, it is easy to check that for the transmission parameter $p \in [\sqrt{2\mu }\widetilde{\omega}_1, \,\sqrt{2\mu }\widetilde{\omega}_2]$, the convergence factor $\rho (\widetilde{\omega}_1, p)$ is increasing with respect to $p$,  and  $\rho (\widetilde{\omega}_2, p)$ is decreasing with respect to $p$. Using then the equioscillation principle, the convergence factor can be minimized when its value at ${\omega}_1$ and ${\omega}_2$ are equal, i.e., $\rho(\widetilde{\omega}_1, p^*) = \rho(\widetilde{\omega}_2, p^*)$, which leads to the unique optimized parameter $p^*=\sqrt{2 \mu \widetilde{\omega}_1 \widetilde{\omega}_2}$.
\end{proof}

\begin{theorem}[Optimized transmission parameter: $\mu > 2+\sqrt{3}$]\label{thm:optparav2}
Let us denote by 
\[R_c := \rho(\widetilde{\omega}_c, p) = \rho(\frac{p}{\sqrt{2\mu}}, p) =  \sqrt{ \frac{(\sqrt{2 \mu}-1)^2+1}{(\sqrt{2}+\sqrt{\mu})^2+\mu} \frac{(\sqrt{2} -\sqrt{\mu})^2+\mu}{(\sqrt{2 \mu}+1)^2+1} }, 
\quad 
k_r := \frac{\widetilde{\omega}_2}{\widetilde{\omega}_1},\] 
and introduce two functions of $\mu$,
\[h_1(\mu) := \frac{\mu^2+1+\sqrt{(\mu^2-4 \mu+1)(\mu^2+4 \mu+1)}}{4\mu}, 
\quad 
h_2(\mu) := \frac{(\mu-1)^2 + \delta}{2\mu}.\]
Moreover, we divide the possible range of $p$ into three intervals,
\[\begin{aligned}
I_l &:=[\widetilde{\omega}_1\sqrt{(\mu-1)^2-\delta}, \sqrt{2\mu} \widetilde{\omega}_1], \quad 
I_c := [\sqrt{2\mu} \widetilde{\omega_1}, \sqrt{2\mu} \widetilde{\omega}_2], \\ 
I_r &:= [\sqrt{2\mu} \widetilde{\omega}_2, \widetilde{\omega}_2 \sqrt{(\mu-1)^2+\delta}].
\end{aligned}\]
According to the value of the ratio $k_r$, we have the following three cases:
\begin{enumerate}
\item[(i)] if $k_r > h_2(\mu)$, then one value of the parameter $p$ minimizing the convergence factor is $p^* =\sqrt{2 \mu \widetilde{\omega}_1 \widetilde{\omega}_2} \in I_c$. This optimized parameter $p^*$ is unique when $\rho(\widetilde{\omega}_1, p^*) \geq R_c$. Otherwise, the minimum of the convergence factor is also attained for any $p$ chosen in a closed interval around $p^*$; 

\item[(ii)] if $h_1(\mu) < k_r \leq h_2(\mu)$, the minimum of the convergence factor is attained for any $p$ chosen in a closed interval around $p^*$; 

\item[(iii)] if $k_r \leq h_1(\mu)$, then the minimum is attained with two distinct values $p_l$ and $p_r$, which can be obtained by solving $\rho(\widetilde{\omega}_1, p)=\rho(\widetilde{\omega}_2, p)$ in two intervals $I_l$ and $I_r$ respectively. Furthermore, these two distinct minimizers are the two positive roots of the fourth-order polynomial
\begin{equation}\label{eq:polyfour}
\frac{p^4}{2} + (\mu \widetilde{\omega}_2-\widetilde{\omega}_1)(\widetilde{\omega}_2-\mu \widetilde{\omega}_1)p^2 + 2 \mu^2 \widetilde{\omega}_1^2 \widetilde{\omega}_2^2 = 0.
\end{equation}
\end{enumerate}
\end{theorem}

\begin{proof} 
The main idea is to look at three intervals $I_l$,  $I_c$ and $I_r$ and find the best value of the transmission parameter $p$ in each interval separately. Let us start with the case when $p \in I_c$, where we have the interior local maximizer $\widetilde{\omega}_c = \frac{p}{\sqrt{2 \mu}}$ lying in the interval $[\widetilde{\omega}_1, \widetilde{\omega}_2]$, as shown in Figure~\ref{fig:rhov1} on the left. 
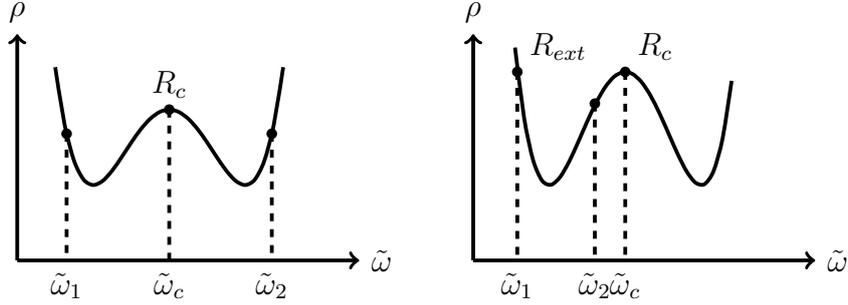
\begin{figure}
\centering
\begin{tikzpicture}
\draw [line width=0.5mm, ->] (0, 0) -- (4.5, 0) node [anchor = west] {$\tilde \omega$};
\draw [line width=0.5mm, ->] (0, 0) -- (0, 3) node [anchor = south] {$\rho$};
\draw [domain = 0.5:3.5, line width=0.5mm, smooth, variable = \x] plot({\x}, {(\x - 1)^2*(\x - 3)^2 + 1});
\draw [dashed, line width=0.5mm] (2, 2) node[anchor = south]{$R_c$} -- (2, 0) node[anchor = north] {$\tilde \omega_c$};
\fill (2, 2) circle (0.07);
\draw [dashed, line width=0.5mm] (0.65, 1.68) -- (0.65, 0) node[anchor = north] {$\tilde \omega_1$};
\fill (0.65, 1.68) circle (0.07);
\draw [dashed, line width=0.5mm] (3.35, 1.68) -- (3.35, 0) node[anchor = north] {$\tilde \omega_2$};
\fill (3.35, 1.68) circle (0.07);
\draw [line width=0.5mm, ->] (6, 0) -- (10.5, 0) node [anchor = west] {$\tilde \omega$};
\draw [line width=0.5mm, ->] (6, 0) -- (6, 3) node [anchor = south] {$\rho$};
\draw [domain = 6.55:9.4, line width=0.5mm, smooth, variable = \x] plot({\x}, {1.5*(\x - 7)^2*(\x - 9)^2 + 1});
\draw [dashed, line width=0.5mm] (7.6, 2.08) -- (7.6, 0) node[anchor = north] {$\tilde \omega_2$};
\fill (7.6, 2.08) circle (0.07);
\draw [dashed, line width=0.5mm] (6.58, 2.5) node[anchor = south west]{$R_{ext}$} -- (6.58, 0) node[anchor = north] {$\tilde \omega_1$};
\fill (6.58, 2.5) circle (0.07);
\draw [dashed, line width=0.5mm] (8, 2.5) node[anchor = south west]{$R_c$} -- (8, 0) node[anchor = north] {$\tilde \omega_c$};
\fill (8, 2.5) circle (0.07);
\end{tikzpicture}
\caption{Illustration of the convergence factor $\rho$ as a function of $\widetilde{\omega}$ with different values of the parameter $p$. Left: $p \in I_c$. Right: $p \in I_r$.}
\label{fig:rhov1}
\end{figure} 
Then using Lemma~\ref{lem:maxomegav1}, the maximum in the min-max problem~\eqref{eq:P1} is given by
\[\max_{\widetilde{\omega}_1 \leq \widetilde{\omega} \leq \widetilde{\omega}_2} \rho(\widetilde{\omega}, p)
= \max \left\{\rho(\widetilde{\omega}_1, p),
\, 
R_c, 
\, 
\rho(\widetilde{\omega}_2, p) \right\}.\]
In this case, we can show that one of the minimal convergence factors can be obtained through the equioscillation property, i.e., $ \rho(\widetilde{\omega}_1, p) =\rho(\widetilde{\omega}_2, p)$, which leads to one of the optimized parameters $p^*$. We also observe that the interior local maximum $R_c$ might be greater than the convergence value at the endpoints with $p = p^*$, i.e., $R_c > \rho(\widetilde{\omega}_1, p^*) = \rho(\widetilde{\omega}_2, p^*)$. In that case, the maximum in the min-max problem~\eqref{eq:P1} is always $R_c$, and from its definition, $R_c$ is constant with respect to $p$. Thus, the minimum of the convergence factor is also attained when we move the parameter $p$ in an interval around $p^*$.

Solving the equality $ \rho(\widetilde{\omega}_1, p) =\rho(\widetilde{\omega}_2, p)$, we obtain a product of two polynomials of $p$,
\begin{equation}\label{eq:polysix}
\big(p^2 - 2\mu \widetilde{\omega}_1 \widetilde{\omega}_2\big)
\Big(\frac{p^4}{2}+(\mu \widetilde{\omega}_2-\widetilde{\omega}_1)(\widetilde{\omega}_2-\mu \widetilde{\omega}_1)p^2+2\mu^2 \widetilde{\omega}_1^2 \widetilde{\omega}_2^2 \Big) = 0.
\end{equation}
For the first polynomial $p^2 - 2\mu \widetilde{\omega}_1 \widetilde{\omega}_2$ in~\eqref{eq:polysix}, there is always one positive root $\sqrt{2 \mu \widetilde{\omega}_1 \widetilde{\omega}_2}$ lying in the interval $I_c$, as $\sqrt{\widetilde{\omega}_1 \widetilde{\omega}_2} \in [\widetilde{\omega}_1, \widetilde{\omega}_2]$. For the second polynomial in~\eqref{eq:polysix}, it is exactly the fourth-order polynomial~\eqref{eq:polyfour}, and we will study in the following its roots according to the value of $k_r$. 
%
%

Now, it remains to look at the optimized parameter $p^*$ in the intervals $I_l$ and $I_r$, and compare the results with those of $I_c$. The situations in these two intervals are very similar, and thus it is sufficient to consider only one case, for instance, $p \in I_r$. In this case, the local maximum point $\widetilde{\omega}_c = \frac{p}{\sqrt{2\mu}} \geq \widetilde{\omega}_2$, and thus lies on the right of the  interval $[\widetilde{\omega}_1,\widetilde{\omega}_2]$, as shown in Figure~\ref{fig:rhov1} on the right. In this case, we obtain once again from Lemma~\ref{lem:maxomegav1} that
\[\max_{\widetilde{\omega}_1 \leq \widetilde{\omega} \leq \widetilde{\omega}_2} \rho (\widetilde{\omega}, p) 
=
\max \left\{\rho(\widetilde{\omega}_1, p),  \rho(\widetilde{\omega}_2, p) \right\}.\]
When $p=\sqrt{2 \mu} \widetilde{\omega}_2$, we have $ \widetilde{\omega}_c =  \widetilde{\omega}_2$, and when $p$ takes other values in $I_r$, $\widetilde{\omega}_c$ moves away from $ \widetilde{\omega}_2$, as shown in Figure~\ref{fig:rhov1} on the right. Substituting $p=\sqrt{2 \mu} \widetilde{\omega}_2$ into~\eqref{eq:rhov1} and using the fact that $k_r = \frac{\widetilde{\omega}_2}{ \widetilde{\omega}_1}$, we obtain for the convergence factor at the endpoints $\widetilde{\omega}_1$ and $\widetilde{\omega}_2$ 
\[\begin{aligned}
\rho(\widetilde{\omega}_1, \sqrt{2 \mu} \widetilde{\omega}_2) &= R_{ext} := \sqrt{ \frac{(\sqrt{2 \mu} k_r-1)^2+1}{(\sqrt{2} k_r+\sqrt{\mu})^2+\mu}\frac{(\sqrt{2} k_r-\sqrt{\mu})^2+\mu}{(\sqrt{2 \mu} k_r+1)^2+1} },  
\\
\rho(\widetilde{\omega}_2, \sqrt{2 \mu} \widetilde{\omega}_2) &= R_c.
\end{aligned}\]
In particular, when $k_r>h_1(\mu)$, we have $R_{ext} > R_c$. To find the optimized parameter $p^*$, we need to compare $R_{ext}$ and $R_c$ to determine the minimum of the convergence factor $\rho$. According to the value of $k_r$, we have the following three cases:
\begin{enumerate}
\item[(i)] if $k_r > h_2(\mu)$, then as $h_2(\mu)>h_1(\mu)$, we have $k_r>h_1(\mu)$, which implies $R_{ext}>R_c$. In this case, the value $\rho(\widetilde{\omega}_1, p)$ increases as $p$ increases in the interval $I_r$, so the convergence factor cannot be improved for $p \in I_r$, and the minimal convergence factor can only be obtained when $p\in I_c$. Furthermore, when $k_r>h_2(\mu)$, there is no positive root for the fourth-order polynomial~\eqref{eq:polyfour}, thus, only one positive root exists for the sixth-order polynomial~\eqref{eq:polysix}, that is $p^*= \sqrt{2 \mu \widetilde{\omega}_1 \widetilde{\omega}_2}\in I_c$. Since the associated $\widetilde{\omega}_c = \frac{p^*}{\sqrt{2 \mu}} = \sqrt{\widetilde{\omega}_1 \widetilde{\omega}_2}$, which falls in the interval $[\widetilde{\omega}_1, \widetilde{\omega}_2]$; then from Lemma~\ref{lem:maxomegav1}, the maximum will be chosen either $R_c$ or $\rho(\widetilde{\omega}_1, p^*) = \rho(\widetilde{\omega}_2, p^*)$. If $\rho(\widetilde{\omega}_1, p^*) = \rho(\widetilde{\omega}_2, p^*) \geq R_c$, then this minimizer $p^*$ is unique for the min-max problem~\eqref{eq:P1}. Otherwise, the maximum is $R_c$, and the minimum of the min-max problem~\eqref{eq:P1} is also $R_c$. As $R_c$ is independent of $p$, it can be attained for any $p$ chosen in a closed interval around $p^*$;

\item[(ii)] if $h_1(\mu) < k_r \leq h_2(\mu)$, we obtain once again $R_{ext}>R_c$. As discussed above in (i), the convergence factor in this case cannot be improved for $p \in I_r$, and the minimal value of the convergence factor will only be obtained when $p\in I_c$. 
Furthermore, the fourth-order polynomial~\eqref{eq:polyfour} has one positive root in $I_c$ if $k_r = h_2(\mu)$, and has two positive roots in $I_c$ if $k_r < h_2(\mu)$. This implies that the sixth-order polynomial~\eqref{eq:polysix} has at least two roots in $I_c$, and we have $\rho(\widetilde{\omega}_1, p) =\rho(\widetilde{\omega}_2, p) \leq R_c$.
Therefore, $R_c$ is the maximum value of $\rho$ for $\widetilde{\omega} \in [\widetilde{\omega}_1,\widetilde{\omega}_2]$. Then the minimum of the convergence factor is attained for any $p$ chosen in a closed interval around $p^*$; 

\item[(iii)] if $k_r \leq h_1(\mu)$, we have $R_{ext} \leq R_c$. Therefore, we can find a unique value $p_r \in  I_r$, $p_r \neq \sqrt{2\mu} \widetilde{\omega}_2$ that satisfies $\rho(\widetilde{\omega}_1, p_r) = \rho(\widetilde{\omega}_2, p_r)$. This then results in the fourth-order polynomial~\eqref{eq:polyfour}, and we have in particular that $R_c > \rho(\widetilde{\omega}_1, p_r) =\rho(\widetilde{\omega}_2, p_r)$. Furthermore, for $p \in  I_r$ and $p \neq \sqrt{2\mu} \widetilde{\omega}_2$, $\widetilde{\omega}_c \notin [\widetilde{\omega}_1, \widetilde{\omega}_2]$, then from Lemma~\ref{lem:maxomegav1}, the maximum will only be chosen between $\rho(\widetilde{\omega}_1, p)$ and $\rho(\widetilde{\omega}_2, p)$, from which we find the minimizer of the min-max problem~\eqref{eq:P1}. In particular, this minimum $\rho(\widetilde{\omega}_1, p_r)$ beats the best convergence factor obtained for $p\in I_c$. 
\end{enumerate}
Based on the similarity of the two intervals $I_r$ and $I_l$, we have respective results for $p\in I_l$. As all possible scenarios have been considered, this completes the proof. 
\end{proof}


\subsection{Local transmission parameter: Version II}\label{sec:v2}
As discussed in Section~\ref{sec:v1}, the choice~\eqref{eq:sgmv1} of the transmission parameter $\sigma_j$ may not be optimal, as it only scales with respect to one diffusion coefficient. To improve it, we consider here a second choice of the local transmission parameters $\sigma_j$
\begin{equation}\label{eq:sgmv2}
\sigma_1 =\sqrt{\nu_2} q, \; \sigma_2=\sqrt{\nu_1} q,  \quad q > 0. 
\end{equation}
This choice now takes into account both diffusion coefficients $\nu_j$ but still with one free parameter $q$. Once again, the convergence of the optimized Schwarz algorithm~\eqref{eq:OSLaplace} is guaranteed by Theorem~\ref{thm:SuffCond} with $q$ positive. For this choice of $\sigma_j$, the convergence factor~\eqref{eq:rhoehatRapprox} becomes
\begin{equation}\label{eq:rhov2}
\rho(\widetilde{\omega}, q) = \sqrt{
\frac {(q- \widetilde{\omega})^2+\widetilde{\omega}^2 }{(q+\mu \widetilde{\omega})^2+\mu^2 \widetilde{\omega}^2} 
\cdot
\frac {(q - \widetilde{\omega})^2+  \widetilde{\omega}^2} {(q+\frac{1}{\mu}\widetilde{\omega})^2+\frac{1}{\mu^2} \widetilde{\omega}^2} 
},
\end{equation}
where $ \mu = \sqrt{\frac{\nu_1}{\nu_2}}$ as before. The related min-max problem~\eqref{eq:P} becomes
\begin{equation}\label{eq:P2}
\min_{ q > 0} \left(\max_{\widetilde{\omega}_1 \leq \widetilde{\omega} \leq \widetilde{\omega}_2} \rho (\widetilde{\omega}, q) \right),
\tag{P2}  
\end{equation}
which turns out to be much easier to analyze compared with the mix-max problem~\eqref{eq:P1}, and we can find a unique optimized transmission parameter $p$.

\begin{theorem}[Optimized transmission parameter: Version II]\label{thm:optparav3}
The unique optimized transmission parameter $q^*$ by solving the min-max problem~\eqref{eq:P2} is given by $q^{\ast} =\sqrt {2\widetilde{\omega}_1 \widetilde{\omega}_2} $.
\end{theorem}

\begin{proof}
The proof follows similar ideas in the proof of Lemma~\ref{lem:restrictpv1} and Lemma~\ref{lem:maxomegav1}. More precisely, we first take the partial derivative of the convergence factor~\eqref{eq:rhov2} with respect to the transmission parameter $q$ and the frequency $\widetilde{\omega}$ respectively,
\[\text{sign}\left( \frac{\partial \rho}{\partial q} \right)= \text{sign}(q^2-2\widetilde{\omega}^2),  
\quad 
\text{sign}\left(\frac{\partial \rho}{\partial \widetilde{\omega}}\right)= \text{sign}(2\widetilde{\omega}^2-q^2).\]
From the partial derivative with respect to $q$ and $\widetilde{\omega}$, we observe that:
\begin{enumerate}
\item[(i)] increasing $q$ will make the convergence factor~\eqref{eq:rhov2} decrease when $q<\sqrt{2} \widetilde{\omega}_1$, and decreasing $q$ will make the convergence factor~\eqref{eq:rhov2} decrease when $q>\sqrt{2} \widetilde{\omega}_2$.  Therefore, we can restrict the range of $q$ to the interval $[\sqrt{2}\widetilde{\omega}_1,  \sqrt{2}\widetilde{\omega}_2]$; 
\item[(ii)] from the partial derivative with respect to the frequency $\widetilde{\omega}$, the convergence factor $\rho(\widetilde{\omega}, q)$ is decreasing for $\widetilde{\omega} \in (\widetilde{\omega}_1,  \frac{q}{\sqrt{2}})$ and is increasing for $\widetilde{\omega} \in (\frac{q}{\sqrt{2}},  \widetilde{\omega}_2)$. This implies that the maximum of the convergence factor $\rho(\widetilde{\omega}, q)$ in the range $[\widetilde{\omega}_1, \widetilde{\omega}_2]$ is $\max\{\rho(\widetilde{\omega}_1, q), \,\rho(\widetilde{\omega}_2, q)\}$; 
\item[(iii)] as for determining the minimum in the min-max problem~\eqref{eq:P2}, we find that $\rho (\widetilde{\omega}_1,  q)$ is increasing, and $ \rho (\widetilde{\omega}_2,  q)$ is decreasing for $q \in [\sqrt{2}\widetilde{\omega}_1,  \sqrt{2}\widetilde{\omega}_2]$.
\end{enumerate}
We can thus conclude that the convergence factor is minimized uniformly by equioscillation, when its value at ${\omega}_1$ and ${\omega}_2$ are equal, i.e., $\rho (\widetilde{\omega}_1,  q^{\ast}) = \rho (\widetilde{\omega}_2,  q^{\ast})$. Solving this equation gives the unique optimized transmission parameter $q^{\ast} =\sqrt {2\widetilde{\omega}_1 \widetilde{\omega}_2}$.
\end{proof}

\subsection{Local transmission parameter: Version III}\label{sec:v3}
In Section~\ref{sec:v2}, we showed a choice~\eqref{eq:sgmv2} taking into account both two diffusion coefficients $\nu_j$ and funnd a unique optimized transmission parameter for the min-max problem~\eqref{eq:P2}. However, we still have only one parameter to tune with this choice for both subdomains $Q_1$ and $Q_2$. More generally, we can consider two transmission parameters,   
\begin{equation}\label{eq:sgmv3}
\sigma_1 = \sqrt{\nu_2} p, 
\ 
\sigma_2 = \sqrt{\nu_1} q,  
\quad  
p, \, q > 0. 
\end{equation}
with two free parameters each for subdomain. The convergence factor~\eqref{eq:rhoehatRapprox} for this choice becomes
\begin{equation}\label{eq:rhov3}
\rho(\widetilde{\omega}, p, q) = \sqrt{
\frac {(p - \widetilde{\omega})^2 + \widetilde{\omega}^2 }{(p + \mu \widetilde{\omega})^2 + \mu^2 \widetilde{\omega}^2} 
\cdot
\frac {(q - \widetilde{\omega})^2 + \widetilde{\omega}^2} {(q + \frac{1}{\mu}\widetilde{\omega})^2 + \frac{1}{\mu^2} \widetilde{\omega}^2} 
}.
\end{equation}
To guarantee convergence of the optimized Schwarz algorithm~\eqref{eq:OSLaplace}, we state next a sufficient condition for the parameters $p$ and $q$ based on Theorem~\ref{thm:SuffCond}.

\begin{corollary}[Sufficient condition]\label{cor:SuffCond}
Suppose that the transmission parameters $p, \, q > 0$ satisfy
\[0<q \leq p  
\quad 
\text{if} \ \nu_1 < \nu_2, 
\quad  
0<p \leq q 
\quad 
\text{if} \ \nu_2 < \nu_1.\]
Then, we have $\rho (\widetilde{\omega}, p, q)<1$ for all $\widetilde{\omega}\in [\widetilde{\omega}_1, \widetilde{\omega}_2]$.
\end{corollary}

The related min-max problem is
\begin{equation}\label{eq:P3}
\min_{p,  q > 0} \left(\max_{\widetilde{\omega}_1 \leq \widetilde{\omega} \leq \widetilde{\omega}_2} \rho (\widetilde{\omega}, p, q) \right). 
\tag{P3} 
\end{equation}
In the following, we consider parameters $p$ and $q$ that satisfy the conditions in Corollary~\ref{cor:SuffCond} to make the optimized Schwarz algorithm~\eqref{eq:OSLaplace} converge. To optimize these two parameters, we follow once again similar steps as in the previous two sections, that is, we first restrict the range for the parameters $(p, q)$ and locate possible values of local maximum point $\widetilde{\omega}$. Then, we analyze how these local maximum points behave when the parameters $(p, q)$ vary. The following result provides the order between $p$ and $q$ in terms of the diffusion coefficient ratio $\mu$.

\begin{lemma}[Order of $p$ and $q$]
If $\mu >1$,  the min-max problem~\eqref{eq:P3} is equivalent to
\[\min_{0<p \leq q} \left(\max_{\widetilde{\omega}_1 \leq \widetilde{\omega} \leq \widetilde{\omega}_2} \rho (\widetilde{\omega}, p, q) \right). \]
If $\mu <1$,  the min-max problem~\eqref{eq:P3} is equivalent to
\[\min_{0<q \leq p} \left(\max_{\widetilde{\omega}_1 \leq \widetilde{\omega} \leq \widetilde{\omega}_2} \rho (\widetilde{\omega}, p, q) \right). \]
\end{lemma}

\begin{proof}
Generally, we can consider to solve the min-max problem in the case $\mu >1$. The other case $\mu < 1$ turns to the case $\mu >1$ by interchanging $p$ and $q$ and replacing $\mu$ by $1/\mu$ in~\eqref{eq:rhov3}. Thus, we assume that $\mu>1$ and $p > q$. The convergence factor is given by~\eqref{eq:rhov3}. Interchanging the values of $p$ and $q$ in~\eqref{eq:rhov3}, this becomes
\[\rho(\widetilde{\omega}, q, p) =\sqrt{ 
\frac {(q- \widetilde{\omega})^2+\widetilde{\omega}^2 }{(q+\mu \widetilde{\omega})^2+\mu^2 \widetilde{\omega}^2} 
\cdot
\frac {(p - \widetilde{\omega})^2+  \widetilde{\omega}^2} {(p+\frac{1}{\mu}\widetilde{\omega})^2+\frac{1}{\mu^2} \widetilde{\omega}^2} 
} .\]
In particular, we have
\[\text{sign}\big(\rho(\widetilde{\omega}, p, q)^2-\rho(\widetilde{\omega}, q, p)^2\big) = \text{sign}\big((\mu-1)(p-q)\big).\]
In the case $\mu>1$ and $p>q$,  we have $\rho(\widetilde{\omega}, p, q)>\rho(\widetilde{\omega}, q, p)$, meaning that the convergence factor $\rho$ is uniformly improved by interchanging $p$ and $q$. Therefore, when $\mu>1$, it is sufficient to consider the parameters $p\leq q$. 
\end{proof}

From now on, we assume that $\mu > 1$ and hence $0<p \leq q$. Then, the conditions in Corollary~\ref{cor:SuffCond} are well satisfied. In this case, we find a similar result as Lemma~\ref{lem:restrictpv1}.

\begin{lemma}[Restrict $p$ and $q$]\label{lem:restrictpv3} 
When $\mu > 1$, we can restrict the range of the parameters $p$ and $q$ to the intervals
\[\begin{aligned}
p &\in \big[\widetilde{\omega}_1(\sqrt{\mu^2+1}-(\mu-1)),  
\; 
\widetilde{\omega}_2(\sqrt{\mu^2+1}-(\mu-1))\big],  
\\
q &\in \big[\widetilde{\omega}_1\frac{\sqrt{\mu^2+1}+(\mu-1)}{\mu},  
\;
\widetilde{\omega}_2\frac{\sqrt{\mu^2+1}+(\mu-1)}{\mu}\big].
\end{aligned}\]
\end{lemma}

\begin{proof} 
Taking a partial derivative of the convergence factor~\eqref{eq:rhov3} with respect to the transmission parameters $p$ and $q$, we find
\[\begin{aligned}
\text{sign}\left( \frac{\partial \rho}{\partial p} \right) &= \text{sign}\left(p^2 + 2p(\mu-1)\widetilde{\omega} - 2\mu \widetilde{\omega}^2\right)\\
&=\left\{\begin{aligned}
&\text{positive}, &&\text{if} \ p> \widetilde{\omega}\big(\sqrt{\mu^2+1}-(\mu-1)\big), \\
&\text{negative}, &&\text{if} \ p< \widetilde{\omega}\big(\sqrt{\mu^2+1}-(\mu-1)\big).
\end{aligned}\right. 
\\
\text{sign}\left( \frac{\partial \rho}{\partial q} \right) &= \text{sign}\left(\mu q^2 - 2q(\mu-1) \widetilde{\omega} - 2\widetilde{\omega}^2\right)\\ 
&=\left\{\begin{aligned}
&\text{positive}, &&\text{if} \ q> \widetilde{\omega}\frac{\sqrt{\mu^2+1}+(\mu-1)}{\mu}, \\
&\text{negative}, &&\text{if} \ q< \widetilde{\omega}\frac{\sqrt{\mu^2+1}+(\mu-1)}{\mu}.
\end{aligned}\right.
\end{aligned}\]
Therefore, when $p < \widetilde{\omega}_1(\sqrt{\mu^2+1}-(\mu-1))$, increasing $p$ improves uniformly the convergence factor $\rho$, while when $p > \widetilde{\omega}_2(\sqrt{\mu^2+1}-(\mu-1))$, decreasing $p$ will improve uniformly the convergence factor $\rho$. Similar arguments hold for the transmission parameter $q$. Therefore, the two restriction intervals follow.
\end{proof}

From the range of $p$ and $q$, we observe that $\frac{p q}{2}$ is actually in the range of $[\widetilde{\omega}_1^2, \widetilde{\omega}_2^2]$. Furthermore, once we restrict the transmission parameters $p$ and $q$, we can find the local maxima of $\widetilde{\omega}$ as in Lemma~\ref{lem:maxomegav1}. Note also that in practice for common choices of $\widetilde{\omega}_j$, where $\widetilde{\omega}_2$ is much larger than $\widetilde{\omega}_1$, we numerically find that the convergence factor $\rho$ behaves as in Figure~\ref{fig:rhov3} when the optimized parameters are obtained. Thus, we consider in the following such convergence behavior and determine the associated optimized parameter pair $(p, q)$.
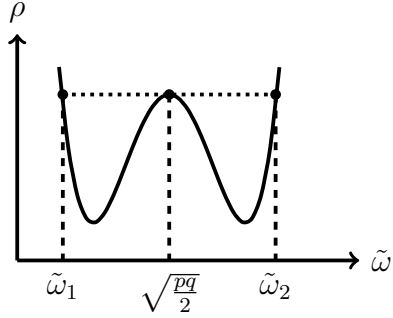
\begin{figure}
\centering
\begin{tikzpicture}
\draw [line width=0.5mm, ->] (0, 0) -- (4.5, 0) node [anchor = west] {$\tilde \omega$};
\draw [line width=0.5mm, ->] (0, 0) -- (0, 3) node [anchor = south] {$\rho$};
\draw [domain = 0.55:3.45, line width=0.5mm, smooth, variable = \x] plot({\x}, {1.7*(\x - 1)^2*(\x - 3)^2 + 0.5});
\draw [dashed, line width=0.5mm] (2, 2.2) -- (2, 0) node[anchor = north] {$\sqrt{\frac{pq}{2}}$};
\fill (2, 2.2) circle (0.07);
\draw [dashed, line width=0.5mm] (0.6, 2.2) -- (0.6, 0) node[anchor = north] {$\tilde \omega_1$};
\fill (0.6, 2.2) circle (0.07);
\draw [dashed, line width=0.5mm] (3.4, 2.2) -- (3.4, 0) node[anchor = north] {$\tilde \omega_2$};
\fill (3.4, 2.2) circle (0.07);
\draw [dotted, line width=0.5mm] (0.6, 2.2) -- (3.4, 2.2);
\end{tikzpicture}
\caption{Illustration of the convergence factor with respect to $\widetilde{\omega}$, when the optimized $p$ and $q$ are obtained.}
\label{fig:rhov3}
\end{figure}

\begin{lemma}[Local maxima of $ \widetilde{\omega}$]\label{lem:maxomegav3} 
For $\widetilde{\omega}\in [\widetilde{\omega}_1, \widetilde{\omega}_2]$,  the maximum of the convergence factor is 
\begin{equation}\label{eq:maxomega}
\max_{\widetilde{\omega}_1 \leq \widetilde{\omega} \leq \widetilde{\omega}_2} \rho(\widetilde{\omega}, p, q) = \max\left\{\rho(\widetilde{\omega}_1, p, q), \; \rho(\sqrt{\frac{pq}{2}}, p, q), \; \rho(\widetilde{\omega}_2, p, q)\right\}.
\end{equation}  
\end{lemma}

\begin{proof} 
Taking a partial derivative of~\eqref{eq:rhov3} with respect to $\widetilde{\omega}$, we get
\begin{equation}\label{eq:DrhoDomega}
\begin{aligned}
\text{sign}\left( \frac{\partial \rho}{\partial \widetilde{\omega}} \right)
&=\text{sign}\Big( (2\widetilde{\omega}^2 - pq) \times\\
&\big( \widetilde{\omega}^2 + \frac{(\mu -1)(\gamma \mu -1)-\sqrt{(\mu^2 +1)(\gamma^2 \mu^2 +1)}}{2 \mu}p\widetilde{\omega} + \frac{\gamma p^2}{2} \big)\Big),
\end{aligned}
\end{equation}  
where we introduced the ratio $\gamma := \frac{q}{p}$. When the first polynomial of $\widetilde{\omega}$ in~\eqref{eq:DrhoDomega} equals zero, i.e., $2\widetilde{\omega}^2 - p q = 0$, we obtain that $\widetilde{\omega} = \sqrt{\frac{p q}{2}}$. To study whether this value is a local maximum point for $\widetilde{\omega} \in [\widetilde{\omega}_1, \widetilde{\omega}_2]$, we need to know the sign of the second polynomial in~\eqref{eq:DrhoDomega} near the point $\widetilde{\omega} = \sqrt{\frac{p q}{2}}$. Using the ratio $\gamma$, we have $\widetilde{\omega}^2 = \frac{\gamma p^2}{2}$. Substituting this into the second polynomial of $\widetilde{\omega}$ in~\eqref{eq:DrhoDomega}, we find
\begin{equation}\label{eq:Drhoomegapart}
\gamma p^2 + \sqrt{\frac{\gamma}{2}}\frac{(\mu -1)(\gamma \mu -1)-\sqrt{(\mu^2 +1)(\gamma^2 \mu^2 +1)}}{2 \mu}p^2.
\end{equation} 
Supposing that~\eqref{eq:Drhoomegapart} is nonnegative, we get
\[\frac{(\mu -1)(\gamma \mu -1)-\sqrt{(\mu^2 +1)(\gamma^2 \mu^2 +1)}}{2 \mu} \geq -\sqrt{2 \gamma}.\] 
We can then bound the second polynomial in~\eqref{eq:DrhoDomega} by
\[\begin{aligned}
\widetilde{\omega}^2 + \frac{(\mu-1)(\gamma\mu-1) - \sqrt{(\mu^2+1)(\gamma^2 \mu^2+1)}}{2 \mu}p\widetilde{\omega} + \frac{\gamma p^2}{2}  & \geq \\
\widetilde{\omega}^2 - \sqrt{2 \gamma}  p\widetilde{\omega} + \frac{\gamma p^2}{2} &= ( \widetilde{\omega}-\frac{\sqrt{2 \gamma} p }{2})^2  \geq 0.
\end{aligned}\]
This implies that the second polynomial in~\eqref{eq:DrhoDomega} is nonnegative, and the sign of the partial derivative only depends on the first polynomial in~\eqref{eq:DrhoDomega}, that is, $\rho(\widetilde{\omega}, p, q)$ is decreasing for $\widetilde{\omega} \in [\widetilde{\omega}_1, \sqrt{\frac{p q}{2}}]$ and is increasing for $\widetilde{\omega} \in [\sqrt{\frac{p q}{2}}, \widetilde{\omega}_2]$. This contradicts the fact that the convergence $\rho$ behaves as in Figure~\ref{fig:rhov3}. Therefore, the equation~\eqref{eq:Drhoomegapart} is negative, and the second polynomial in~\eqref{eq:DrhoDomega} is also negative when $\widetilde{\omega}^2= \frac{p q}{2}$. For this reason, the convergence factor $\rho$ has a local maximum in $\widetilde{\omega}$ at $\sqrt{\frac{p q}{2}}$. According to the range of the transmission parameters $p$ and $q$, we have $\sqrt{\frac{p q}{2}}\in [\widetilde{\omega}_1, \widetilde{\omega}_2]$. Therefore,  the maximum value of the convergence factor $\rho(\widetilde{\omega}, p, q)$ for $\widetilde{\omega}\in[\widetilde{\omega}_1, \widetilde{\omega}_2]$ is given by~\eqref{eq:maxomega}.
\end{proof}

With the help of Lemma~\ref{lem:restrictpv3} and Lemma~\ref{lem:maxomegav3}, we obtain a similar result as Theorem~\ref{thm:optparav1} and Theorem~\ref{thm:optparav3} that the optimized transmission parameter pair $(p^*, q^*)$ can be obtained by an equioscillation of these three local maxima.

\begin{theorem}[Optimized transmission parameters: Version III] 
When $\mu> 1$, the unique minimizer pair $(p^*,  q^*)$ of Problem~\eqref{eq:P3} is the solution of the system of the two equations
\[\rho(\widetilde{\omega}_1, \, p^*, \, q^*) = \rho(\widetilde{\omega}_2, \, p^*, \, q^*),
\quad
\rho(\widetilde{\omega}_1, \, p^*, \, q^*) = \rho(\sqrt{\widetilde{\omega}_1 \widetilde{\omega}_2}, \, p^*, \, q^*).\]
\end{theorem}

\begin{proof} 
According to the equioscillation principle, we need to have at the endpoints of the frequency $\widetilde{\omega}$ that $\rho(\widetilde{\omega}_1, p, q) = \rho(\widetilde{\omega}_2, p, q)$ to acquire the minimum of the convergence factor $\rho$. After some algebraic simplification, we obtain $p q=2 \widetilde{\omega}_1 \widetilde{\omega}_2$. This then enables us to reduce the range of the parameter to $p \in I_p := [\widetilde{\omega}_1(\sqrt{\mu^2+1}-(\mu-1)),  \sqrt{2 \widetilde{\omega}_1 \widetilde{\omega}_2}]$, and the min-max problem~\eqref{eq:P3} becomes
\[\min_{p\in I_p} \big( \max\{R_1(p), \; R_c(p)\} \big), 
\
R_1(p):=\rho(\widetilde{\omega}_1, p, \frac{2 \widetilde{\omega}_1 \widetilde{\omega}_2}{p}), 
\,
R_c(p):=\rho(\sqrt{\widetilde{\omega}_1 \widetilde{\omega}_2}, p, \frac{2 \widetilde{\omega}_1 \widetilde{\omega}_2}{p}).\]
Using once again the equioscillation principle, the optimized parameters $p^*$ can be found when $R_1(p) = R_c(p)$ for $p\in I_p$, which can be reduced to the equation
\begin{equation}\label{eq:V3solution} 
\frac{(p-\widetilde{\omega}_1)^2+\widetilde{\omega}_1^2}{(p+\mu \widetilde{\omega}_1)^2+\mu^2 \widetilde{\omega}_1^2}
\cdot 
\frac{(p-\widetilde{\omega}_2)^2+\widetilde{\omega}_2^2}{(p+ \mu \widetilde{\omega}_2)^2+\mu^2\widetilde{\omega}_2^2}
=
\Big(\frac{(p- \sqrt{\widetilde{\omega}_1 \widetilde{\omega}_2})^2+\widetilde{\omega}_1 \widetilde{\omega}_2}{(p+\mu \sqrt{\widetilde{\omega}_1 \widetilde{\omega}_2})^2+\mu^2 \widetilde{\omega}_1 \widetilde{\omega}_2}\Big)^2 .
\end{equation}
Solving then this polynomial of $p$, we can identify the optimized transmission parameters. Note that there exist closed forms for the roots of this polynomial. Among all, we can list three simple solutions, that are $0$ and $\pm i\sqrt{2 \mu \widetilde{\omega}_1 \widetilde{\omega}_2}$, the other roots are much more complicated. In practice, when the time step $\Delta t$ is small, the frequency $\widetilde{\omega}_2 = \sqrt{\frac{\pi}{2\Delta t}}$ is much greater than $\widetilde{\omega}_1 = \sqrt{\frac{\pi}{4T}}$. In this case, we can use asymptotic analysis and find an approximate solution $p^* \approx \frac{2 \mu}{\mu-1}\widetilde{\omega}_1$, which lies in the interval $I_p$. Overall, for all the roots, we find one unique real root $p^* \in I_p$, and use once again the fact that $p q=2 \widetilde{\omega}_1 \widetilde{\omega}_2$ to find $q^*$, and this completes the proof. 
\end{proof}

\begin{remark}
To avoid complex and expensive calculation, we can show numerically the graph of~\eqref{eq:V3solution} in Figure~\ref{fig:root} where $p\in I_p$ with a set of $(\widetilde{\omega}_1, \widetilde{\omega}_2, \mu)$. It can be seen that there exists a unique root in~\eqref{eq:V3solution} for $p\in I_p$. Note that the behavior illustrated in Figure~\ref{fig:root} remains similar for all our numerical experiments with different sets of $(\widetilde{\omega}_1, \widetilde{\omega}_2, \mu)$.
\begin{figure}
\centering
\includegraphics[scale=0.25]{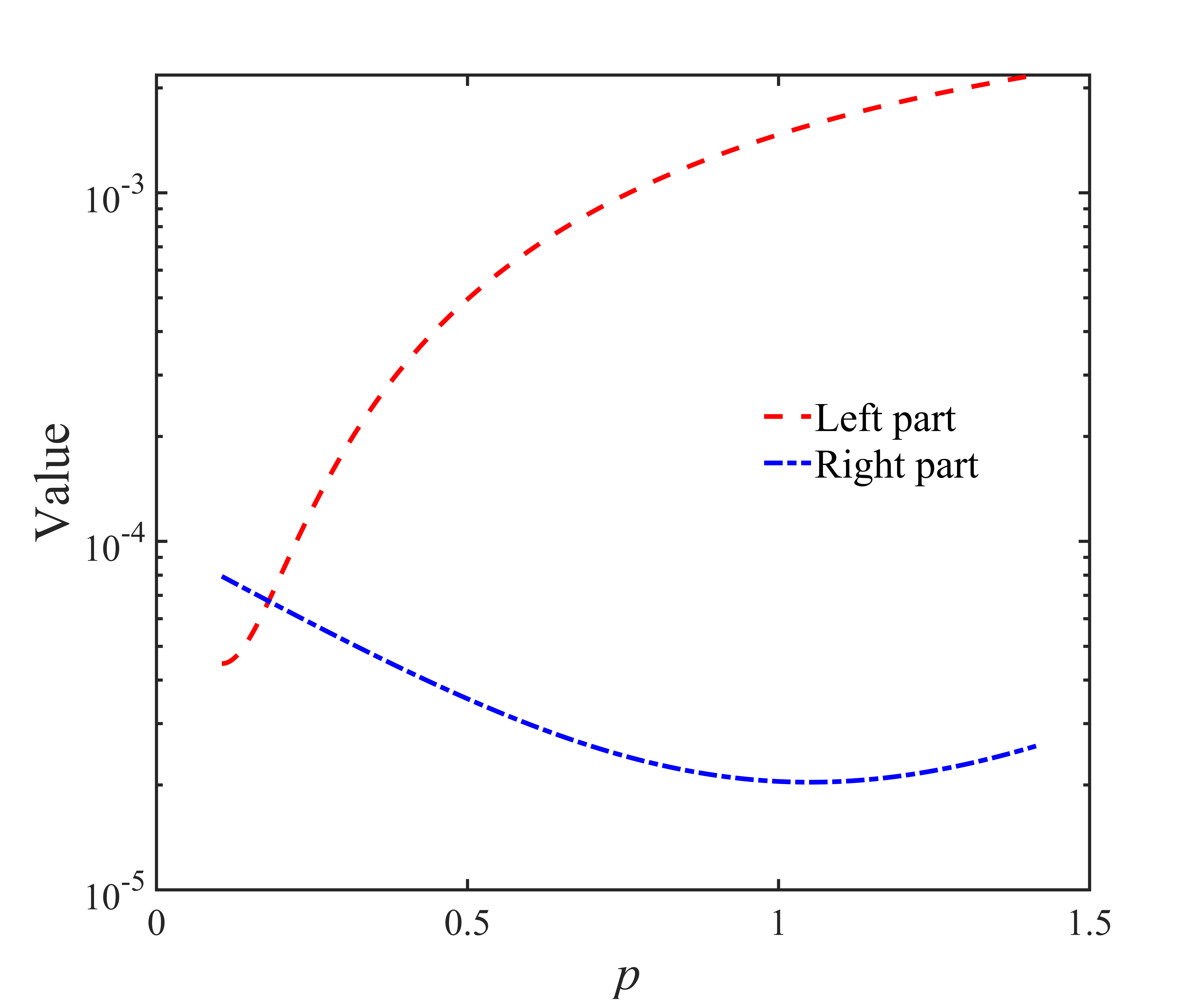}
\caption{Illustration of the left and rights part in~\eqref{eq:V3solution} for $p\in I_p$. }
\label{fig:root}
\end{figure}
\end{remark}


\section{Numerical Experiments}\label{sec:4}
We now show some numerical experiments to compare the performance of the three local approximations of the optimal operator $\sigma_j$ discussed in Section~\ref{sec:3}. For our numerical tests, we consider solving the problem~\eqref{eq:heat} in a one-dimensional space domain $\Omega=(0, 1)$ and for a fixed final time $T=5$. Furthermore, we take a source term $f=0$, a constant initial condition $u_0=20$ and a homogenous Dirichlet boundary condition $g=0$. The space domain $\Omega$ is decomposed into two nonoverlapping subdomains $\Omega _1=(0, \frac{1}{2})$ and $\Omega_2=(\frac{1}{2}, 1)$. In all numerical experiments, the heat diffusion coefficients are $\nu_1=1$ and $\nu_2 = \frac{1}{\mu^2}$, where the ratio $\mu^2 = \frac{\nu_1}{\nu_2}$ is always chosen to be greater than 1. We use a finite element discretization in space with a uniform mesh size $\Delta x$, and a backward Euler discretization in time with a constant time step $\Delta t$. In the Schwarz iteration, we use the $L^{\infty}$ error
\[ e^n:=\|\mathbf{U}-\mathbf{u}^n\|_{\infty}, \]
where $\mathbf{U}$ is the discrete global solution of the problem~\eqref{eq:heat} and $\mathbf{u}^n$ is the combined solution of the subdomains at iteration $n$.

\subsection{Impact of the ratio $\mu$}
We first test the impact of the heat diffusion coefficient ratio $\mu$. For a given mesh size $\Delta x = 1 / 40$ and a time step $\Delta t=1 / 40$, we show in Figure~\ref{fig:2mu} the convergence behavior of the three local transmission conditions for the two different ratios $\frac{\nu_1}{\nu_2}=[10, 10^2]$.
\begin{figure}
\centering
\includegraphics[scale=0.23]{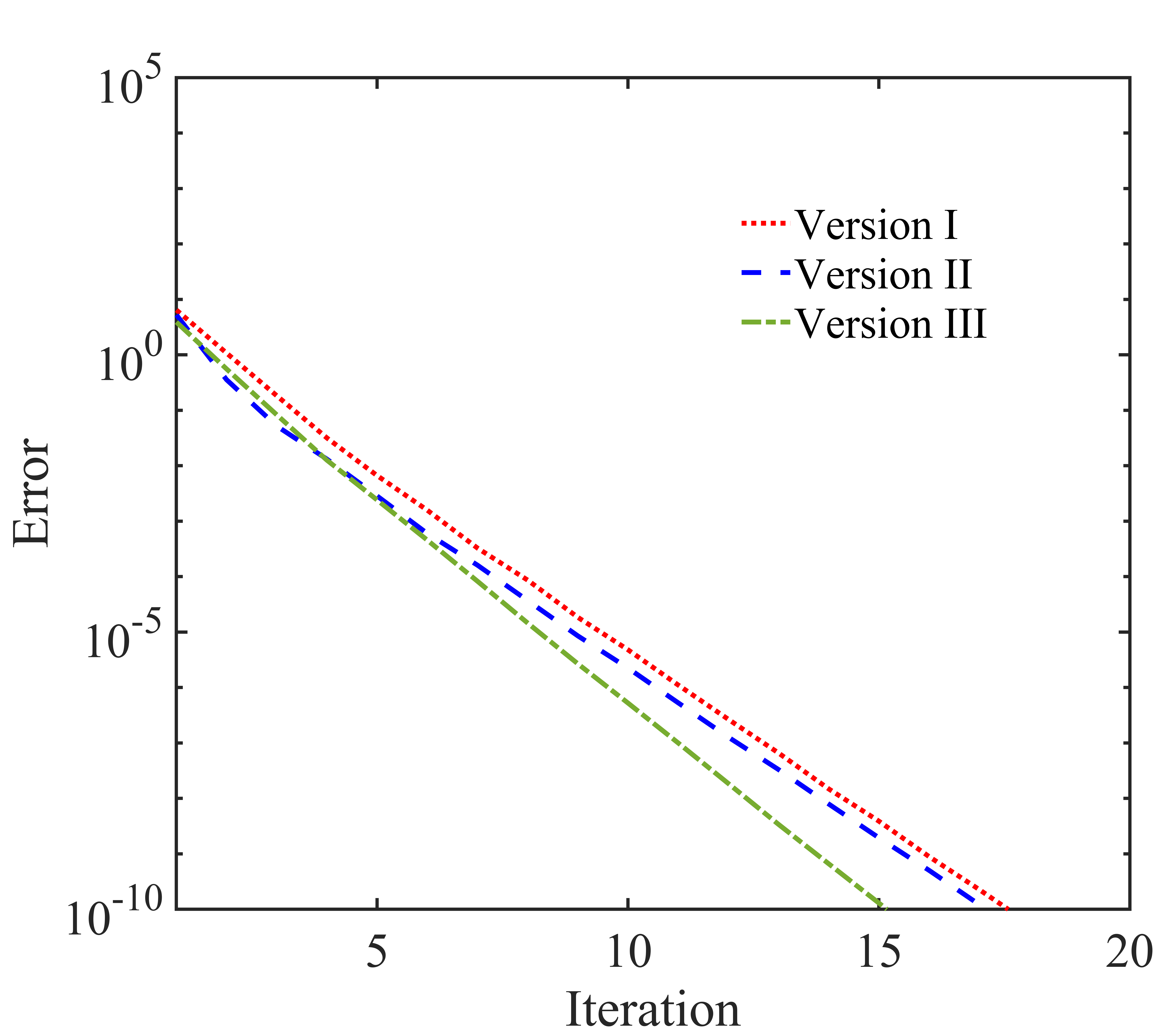}
\includegraphics[scale=0.23]{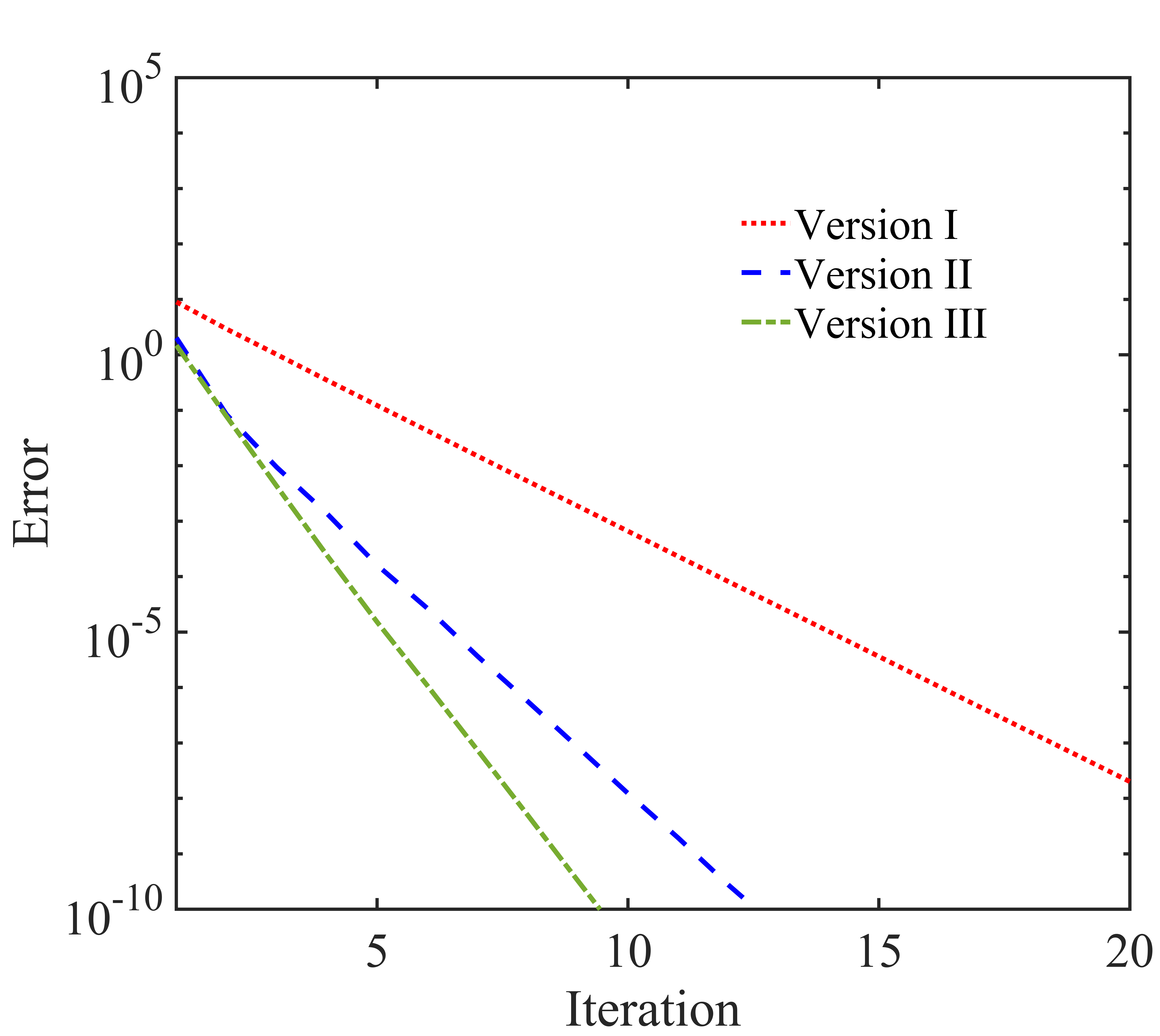}
\caption{Convergence behavior of the three local transmission conditions with a mesh size $\Delta x = \frac{1}{40}$ and a time step $\Delta t=\frac{1}{40}$. Left: $\frac{\nu_1}{\nu_2}=10$. Right: $\frac{\nu_1}{\nu_2}=10^2$.}
\label{fig:2mu}
\end{figure}
We observe that the convergence behavior of Version II and III are slightly better than that of Version I in the case $\mu^2 = 10$, as shown in Figure~\ref{fig:2mu} on the left. However, for the ratio $\mu^2 =10^2$, we observe in Figure~\ref{fig:2mu} on the right that the performance of Version II and III become much better, while Version I becomes less efficient. As expected, the local transmission conditions Version II and Version III are appropriately scaled with respect to both diffusion coefficients $\nu_1$ and $\nu_2$, and thus perform better; but Version I is only scaled with respect to one diffusion coefficient $\nu_2$, thus is less robust when the ratio is changed. Overall, the performance of Version III is the best for the two cases tested.

For this test case, we also show in Figure~\ref{fig:rho} the convergence factor $\rho$ as function of the frequency $\tilde \omega$ of these three versions. Similarly, we observe that Version III yields a much smaller convergence factor compared to the other two versions, which also confirms the convergence behavior observed in Figure~\ref{fig:2mu}.
\begin{figure}
\centering
\includegraphics[scale=0.15]{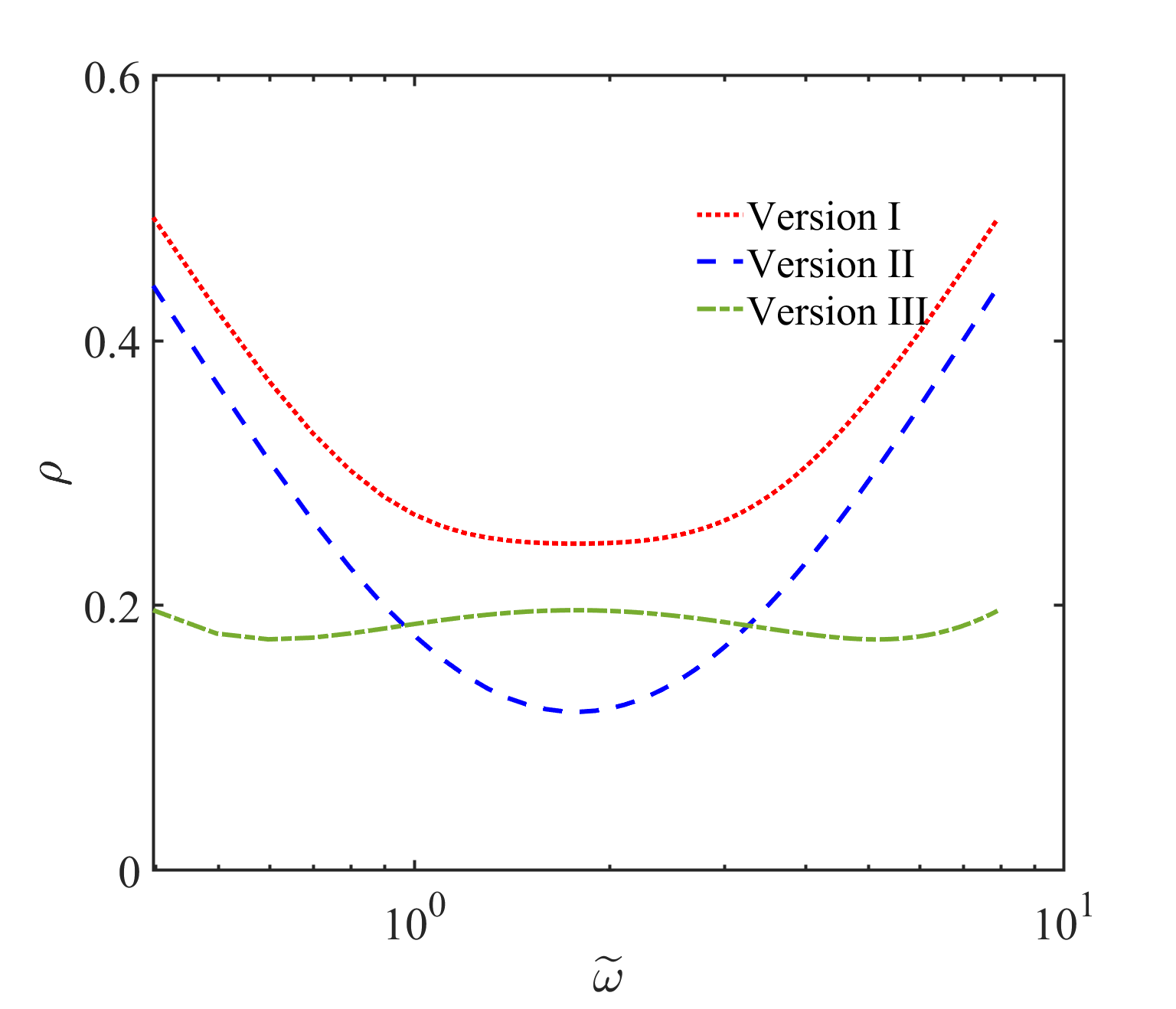}
\includegraphics[scale=0.15]{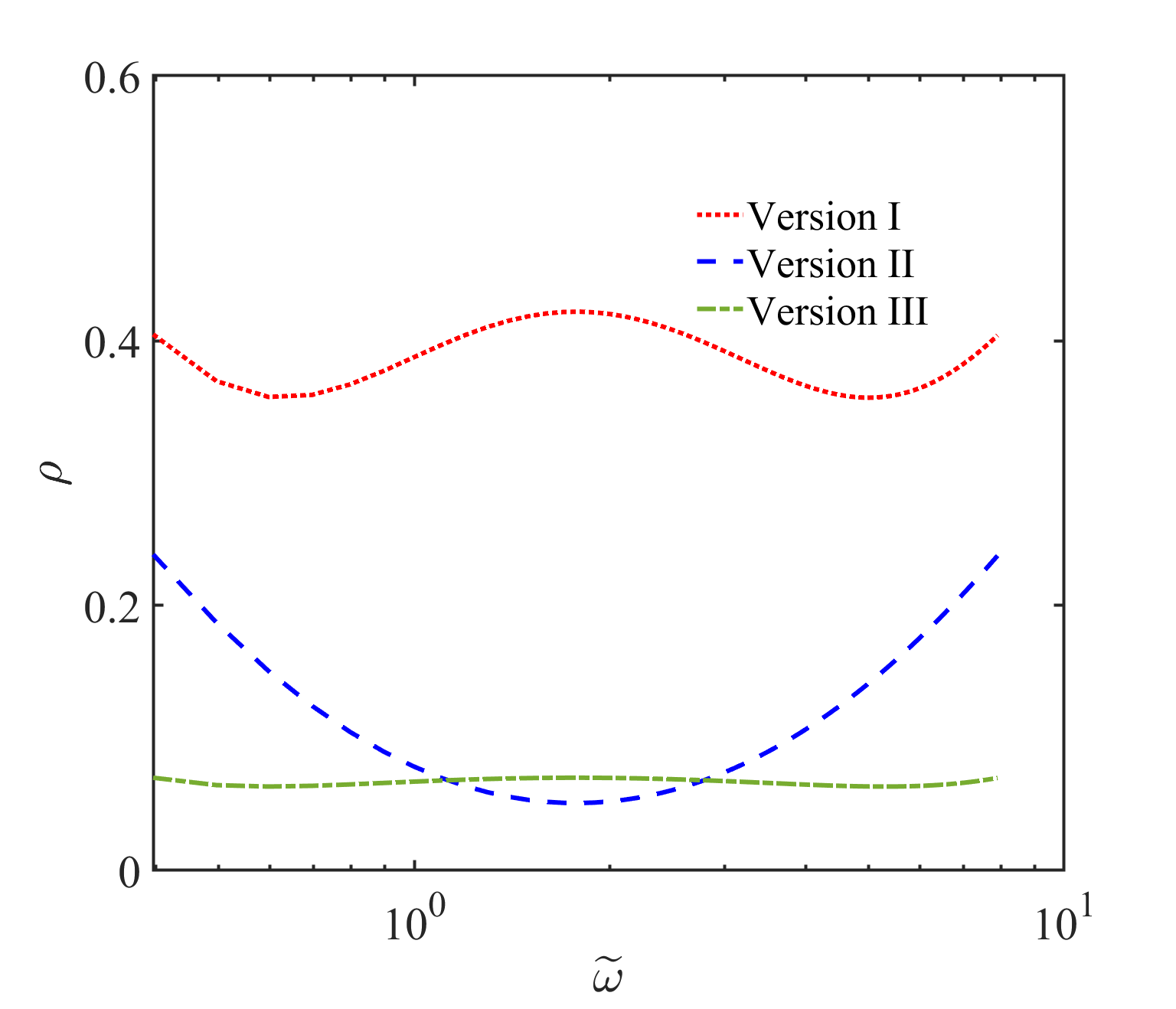}
\caption{Comparison of the convergence factor $\rho$ with respect to the frequency $\tilde \omega$ for all three versions. Left: $\frac{\nu_1}{\nu_2}=10$. Right: $\frac{\nu_1}{\nu_2}=10^2$.}
\label{fig:rho}
\end{figure}

To get better insights into the impact of the ratio, we keep the mesh size $\Delta x = 1 / 40$ and the time step $\Delta t = 1 / 40$ and vary the diffusion coefficients ratio $\mu^2$. Table~\ref{tab:4ratio} shows the number of iterations needed to reach a tolerance of $10^{-8}$ for the three versions when the diffusion coefficient ratio increases. 
\begin{table}\label{tab:4ratio}
\centering
\caption{Number of iterations to reach a tolerance of $10^{-8}$ for four ratios $\frac{\nu_1}{\nu_2}$. }
\begin{tabular}{|c|c|c|c|}
\hline
$\mu^2=\frac{\nu_1}{\nu_2}$ & Version I & Version II & Version III \\ 
\hline
$10^1$ & 15 & 14 & 13 \\ 
$10^2$ & 21 & 11 & 8 \\ 
$10^3$ & 39 & 9 & 6 \\ 
$10^4$ & 169 & 9 & 6 \\ 
\hline
\end{tabular}
\end{table}
We observe once again that the convergence behavior of Version II and III is better than Version I. In particular, as Version I is only scaled with respect to $\nu_2$ for both local transmission parameters, that is, $\sigma_1=\sigma_2=\sqrt{\nu_2}p$, thus when the ratio $\mu$ increases, they cannot take into account this change accordingly in each subdomain, and become much worse for large ratios. On the contrary, both Version II and III are scaled with respect to two diffusion coefficients $\nu_1$ and $\nu_2$, thus able to handle much easier when changing the coefficient ratio. They become even much more efficient and robust for a large coefficient ratio $\mu$. Among all, Version III outperforms the others for all tested cases in Table~\ref{tab:4ratio}.

\subsection{Influence of the time step $\Delta t$}
Next, we test the impact of the time step $\Delta t$, which will influence the high frequency value $\omega_{\max} =\pi / \Delta t$, thus changes the range of the frequency $\omega$. We keep the same mesh size $\Delta x = 1 / 40$ and consider two different diffusion ratios $\frac{\nu_1}{\nu_2} = 10$ and $\frac{\nu_1}{\nu_2} = 10^3$. We investigate here the impact of the time step $\Delta t$ in both two cases. The convergence behavior for the four different time steps $\Delta t= [\frac{1}{20}, \frac{1}{40}, \frac{1}{80}, \frac{1}{160}]$ is illustrated in Figure~\ref{fig:dt}. 
\begin{figure}
\begin{center}
\includegraphics[scale=0.15]{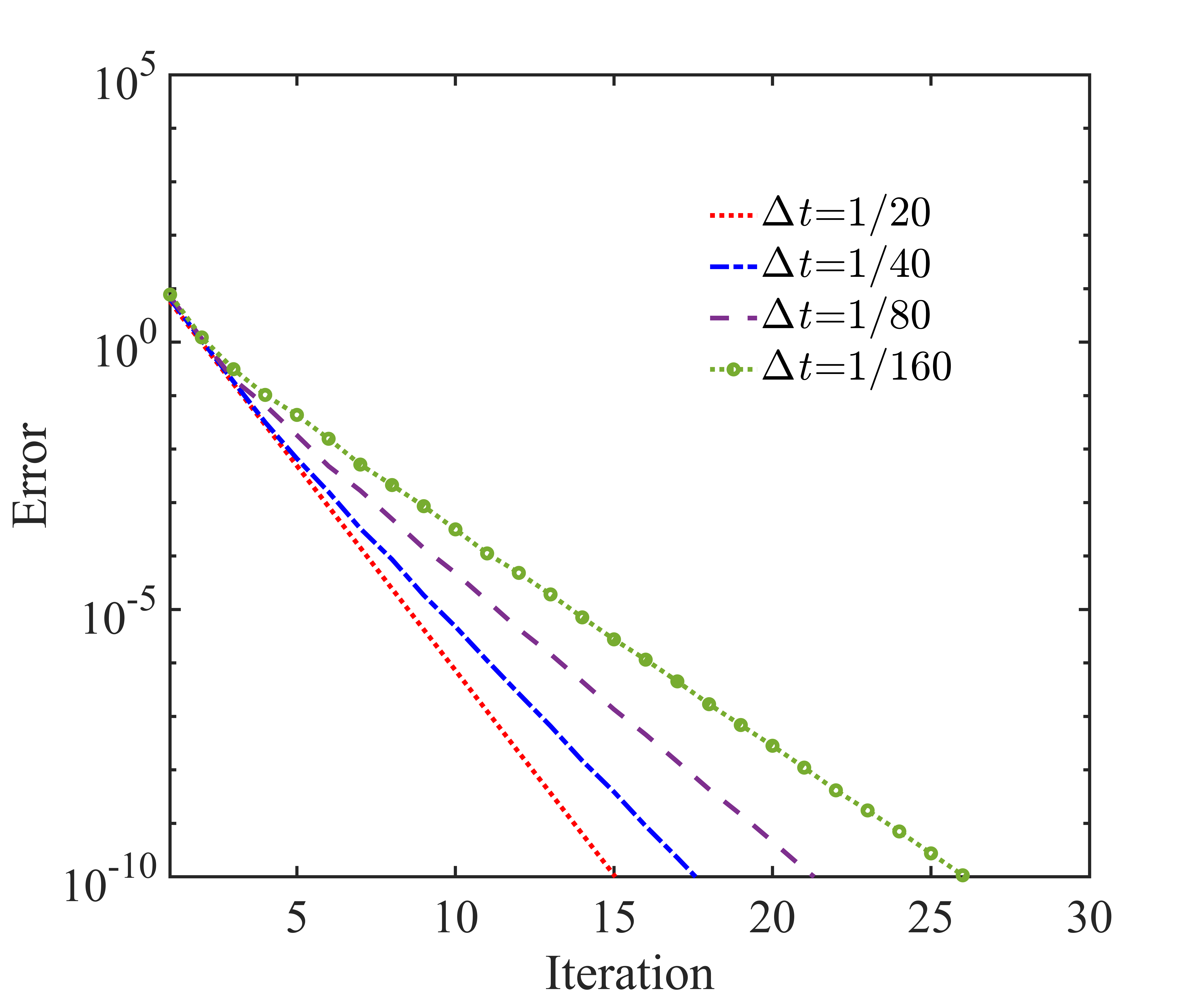}
\includegraphics[scale=0.15]{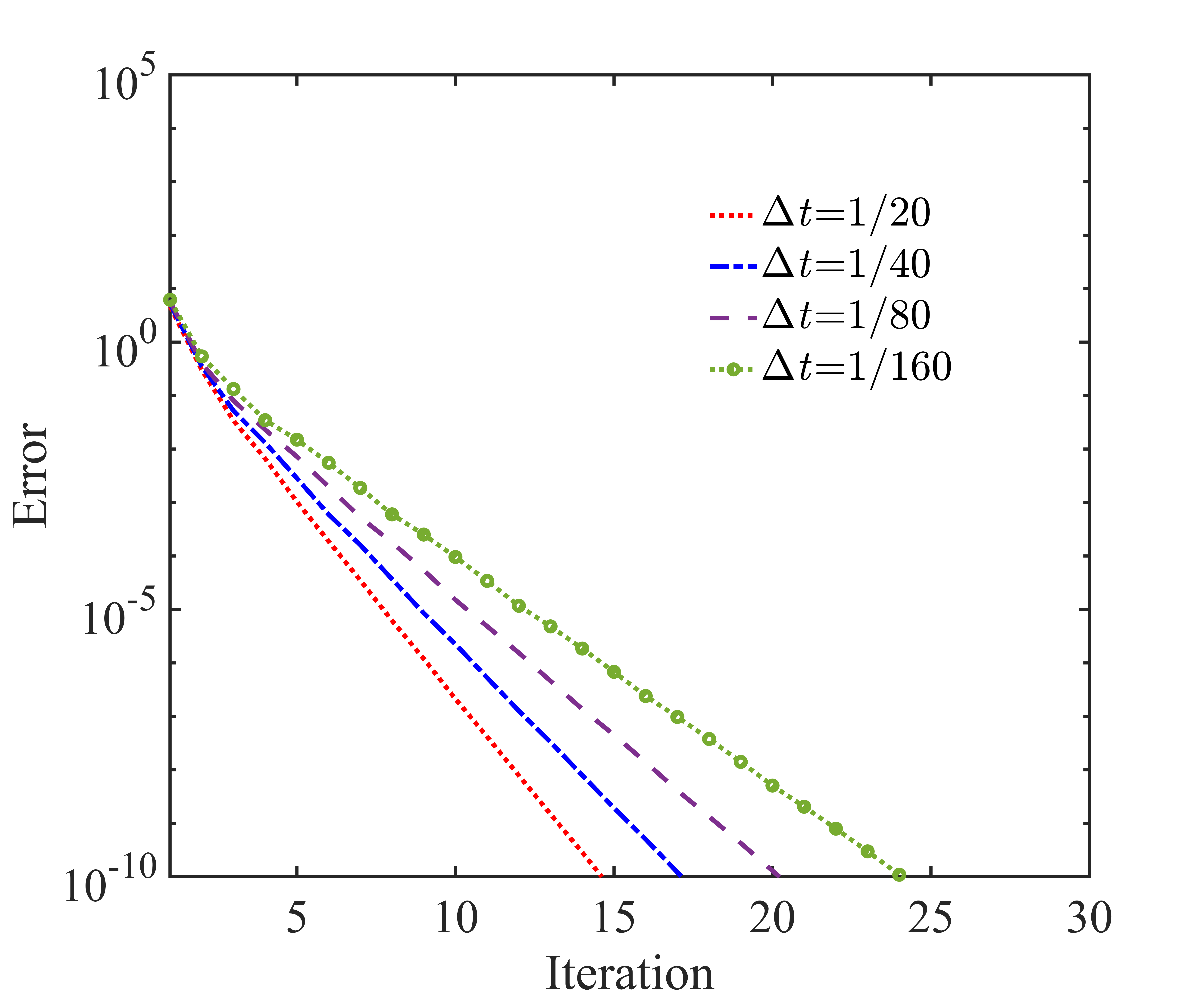} 
\includegraphics[scale=0.15]{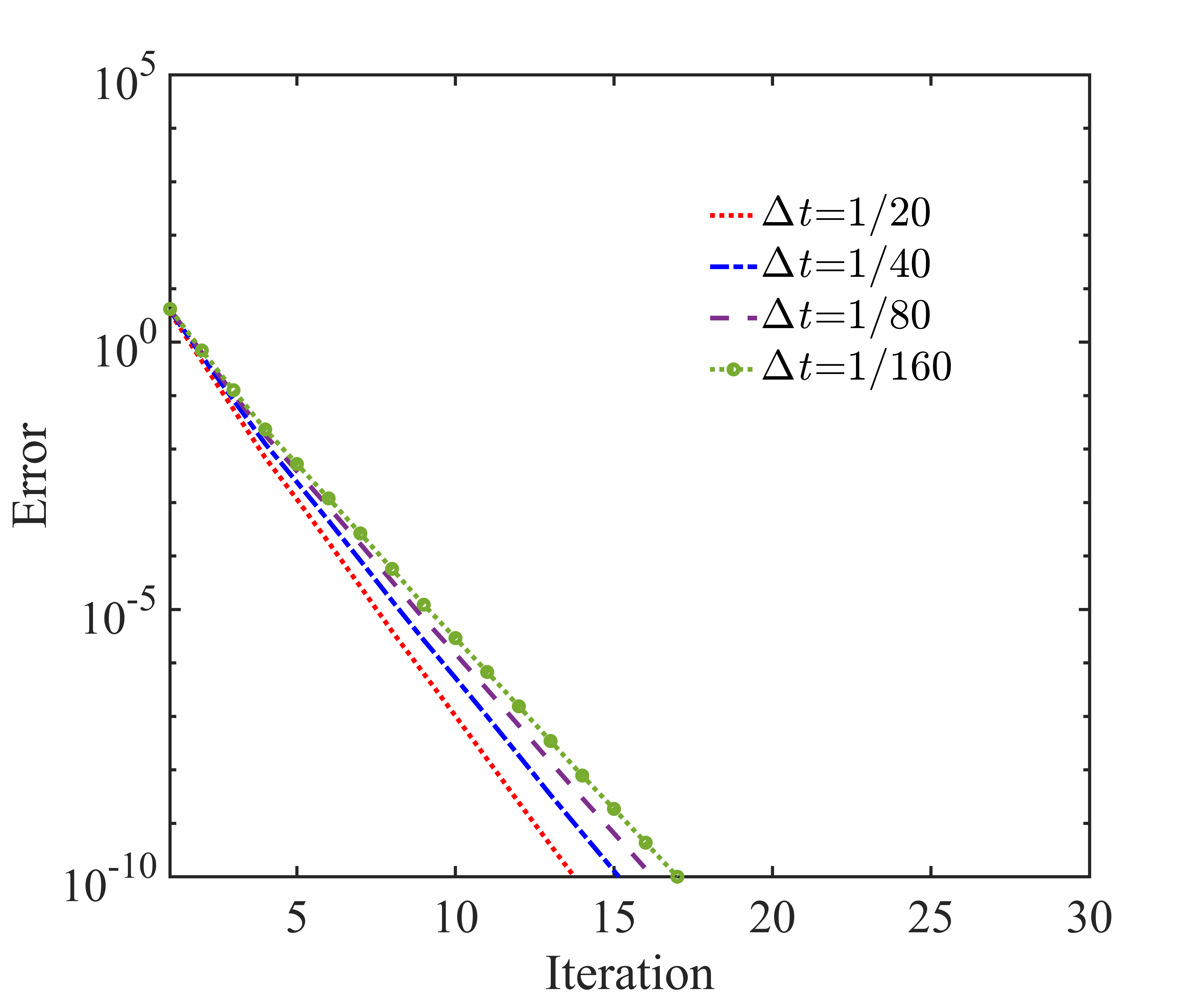} 
\includegraphics[scale=0.15]{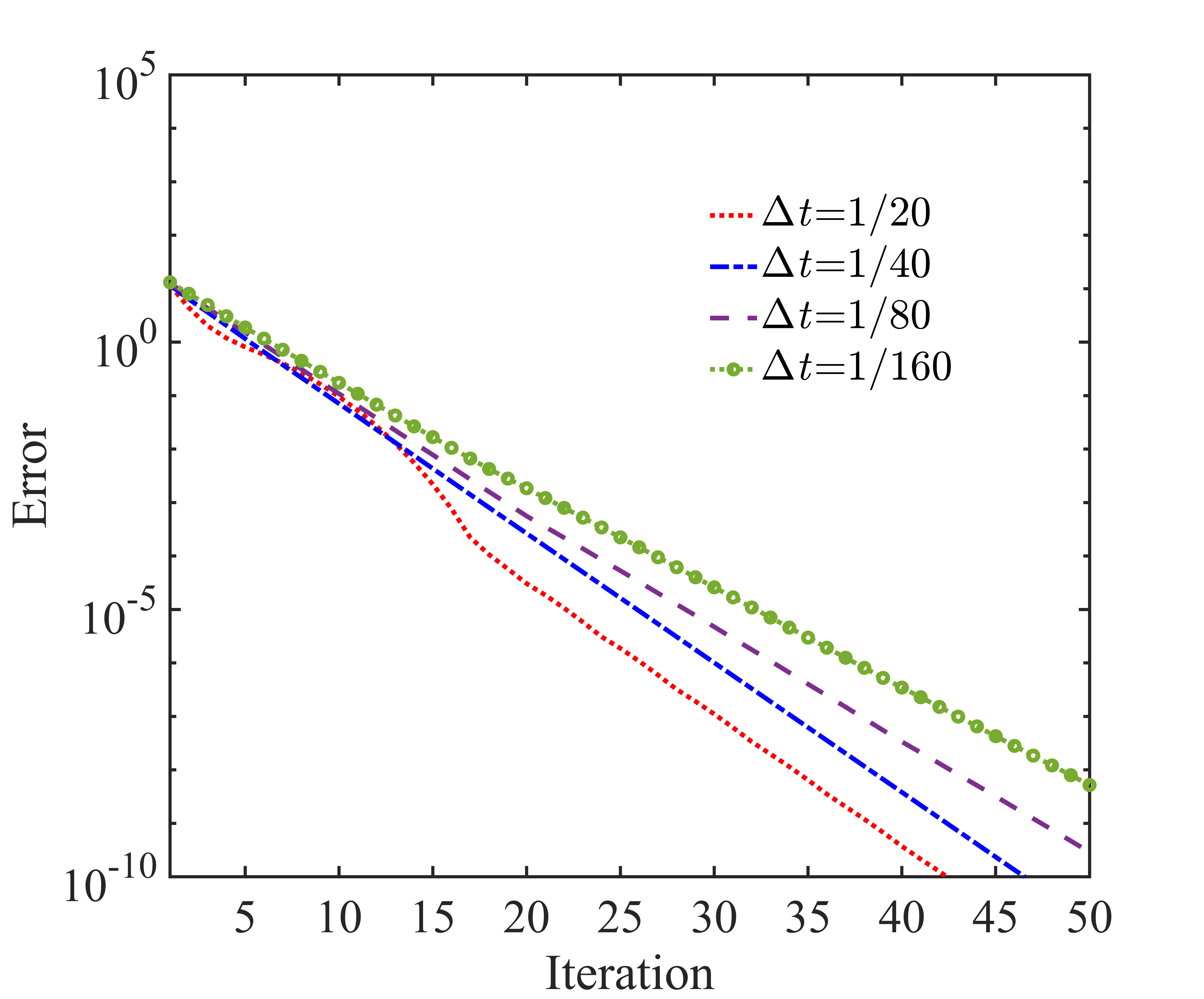}
\includegraphics[scale=0.15]{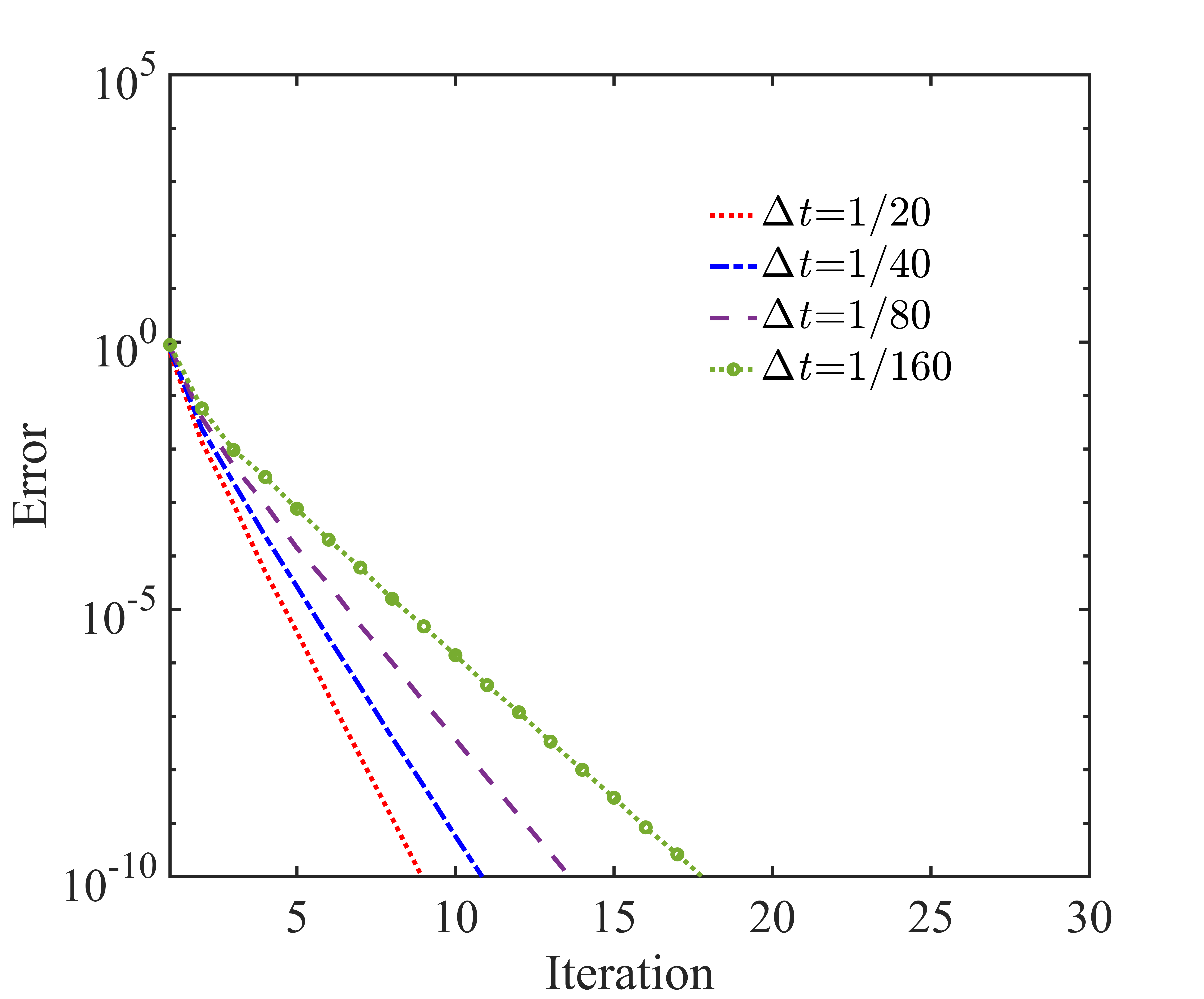} 
\includegraphics[scale=0.15]{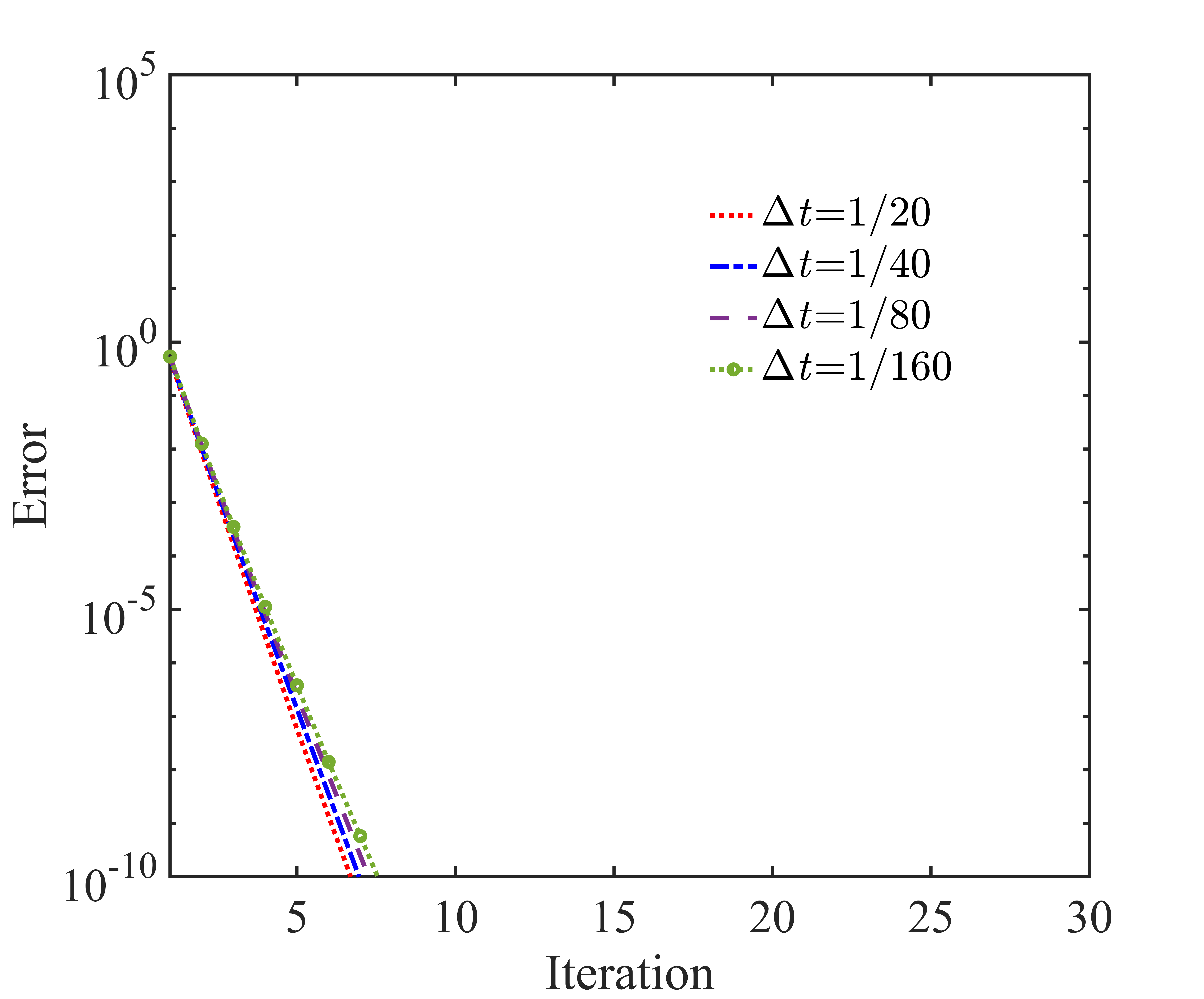} 
\end{center}
\caption{Convergence behavior of the three local transmission conditions with a given mesh size $\Delta x = \frac{1}{40}$ and four different time steps $\Delta t= [\frac{1}{20}, \frac{1}{40}, \frac{1}{80}, \frac{1}{160}]$. Top: $\frac{\nu_1}{\nu_2}=10$. Bottom: $\frac{\nu_1}{\nu_2}=10^3$. Left: Version I. Middle: Version II. Right: Version III. }
\label{fig:dt}
\end{figure}
Generally speaking, we observe that the convergence becomes less efficient when the time step $\Delta t$ decreases. In particular, the convergence of Version I and II deteriorates for small time step as shown in Figure~\ref{fig:dt} on the left and in the middle, whereas the performance of Version III varies very little when decreasing the time step especially for large diffusion ratio. Among all the tested cases, the convergence of Version III is more stable as shown in Figure~\ref{fig:dt} on the right.

\subsection{Influence of the mesh size $\Delta x$}
In a similar way, we test now the impact of the mesh size $\Delta x$ in the case of a relatively small ratio $\frac{\nu_1}{\nu_2} = 10$ and a large ratio $\frac{\nu_1}{\nu_2} = 10^3$. We keep the time step $\Delta t = 1 / 40$ and show in Figure~\ref{fig:dx} the convergence behavior for the three different mesh sizes $\Delta x = [\frac{1}{20}, \frac{1}{40}, \frac{1}{80}]$. 
\begin{figure}
\begin{center}
\includegraphics[scale=0.15]{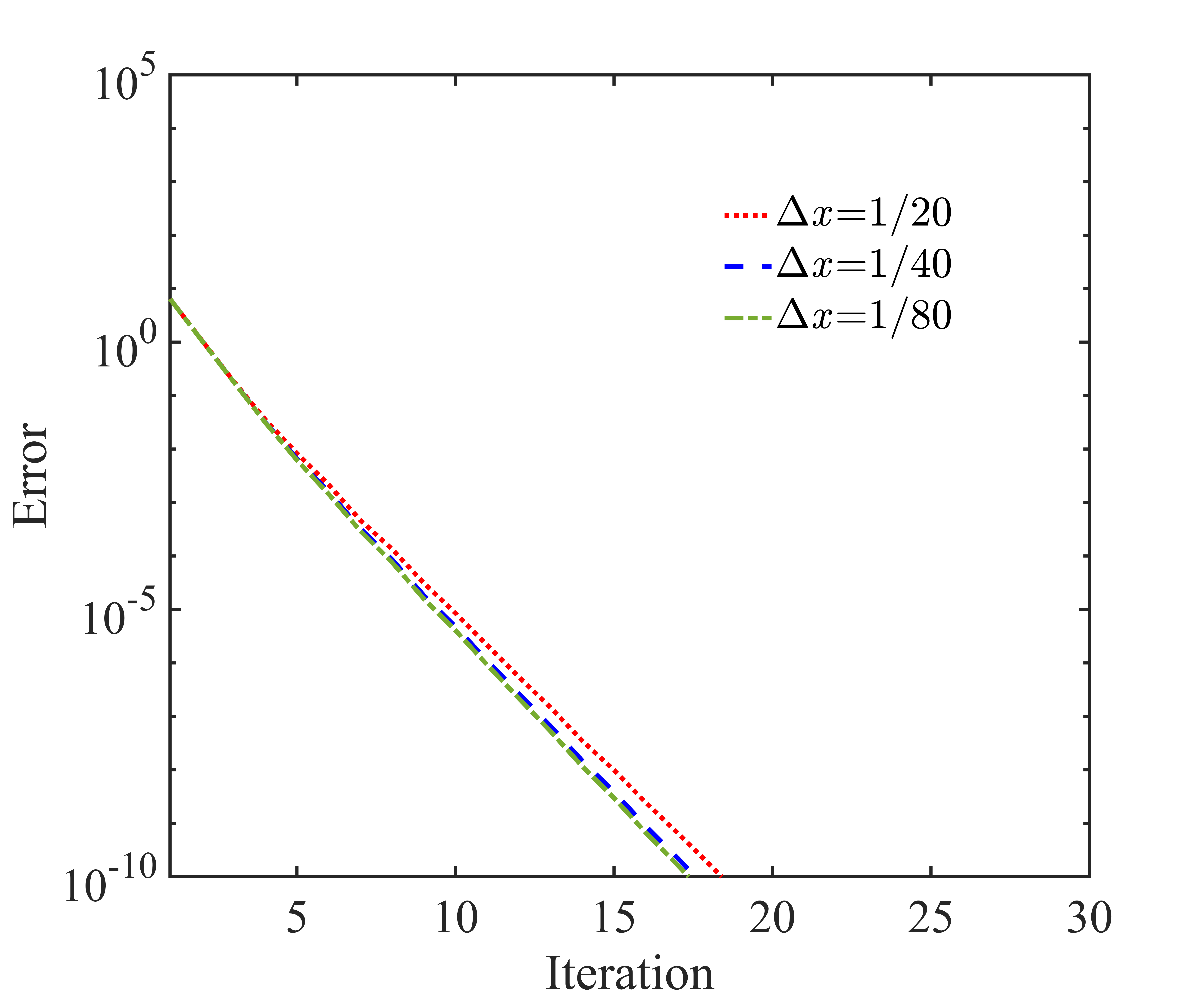}
\includegraphics[scale=0.15]{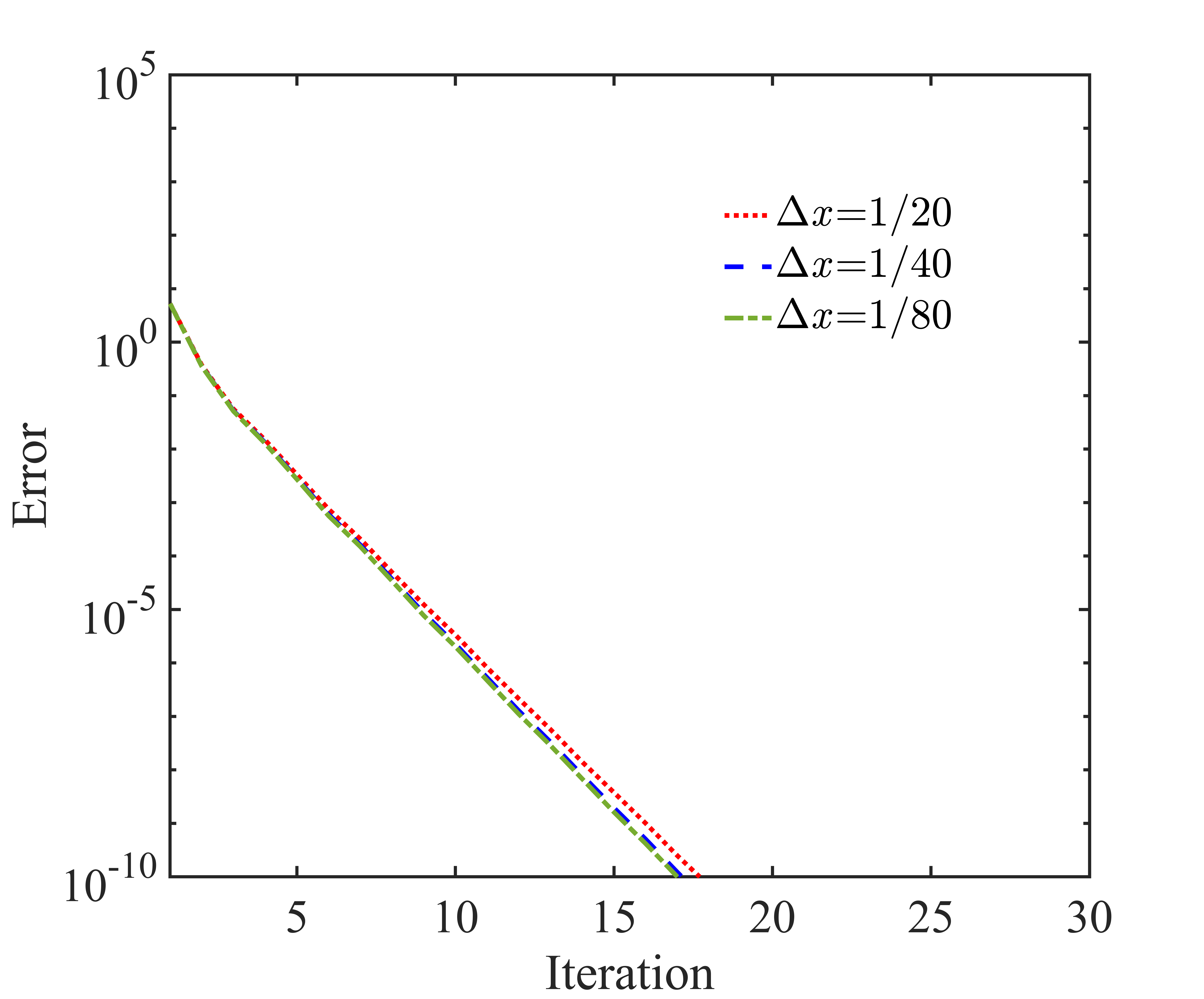} 
\includegraphics[scale=0.15]{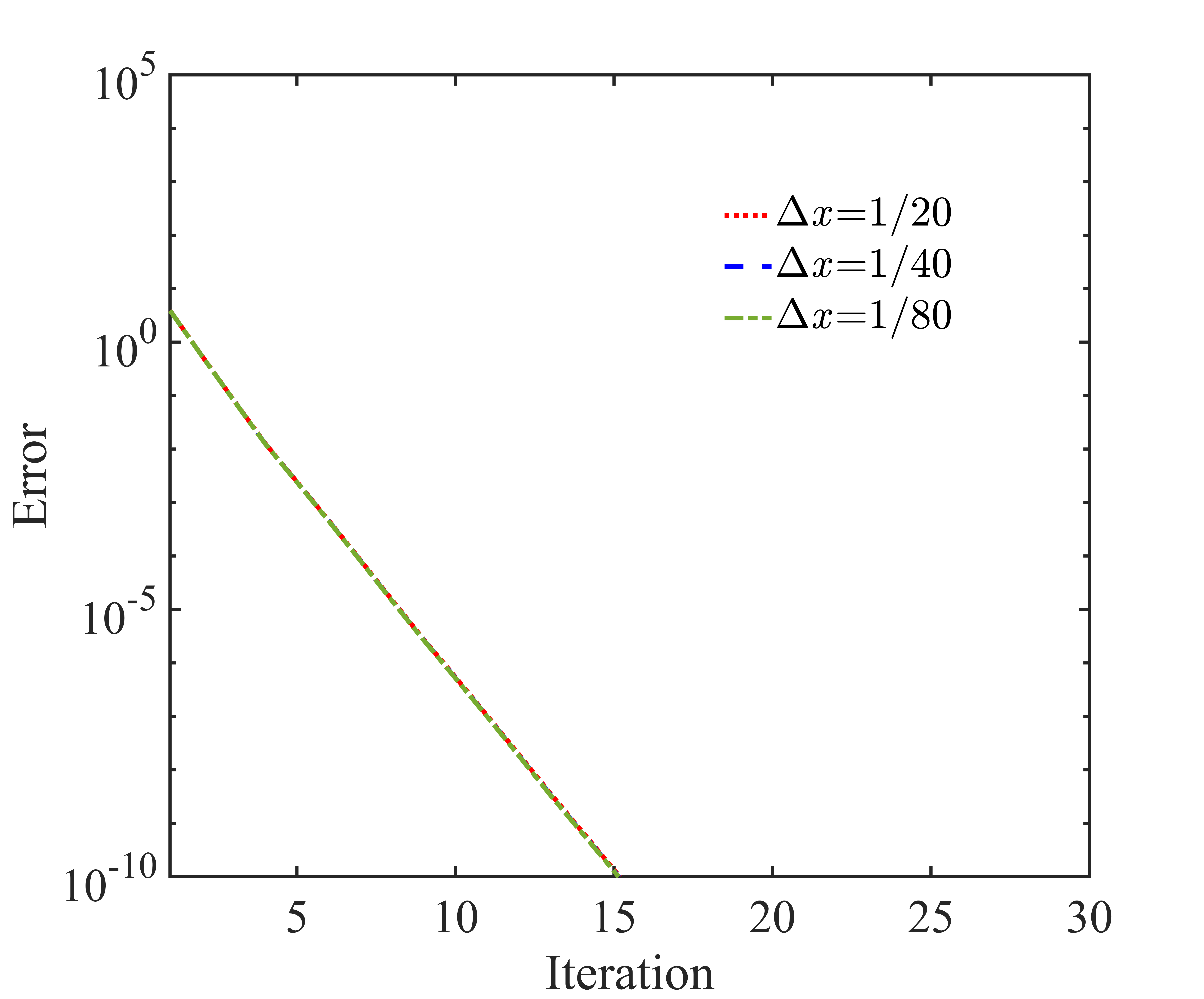}
\includegraphics[scale=0.15]{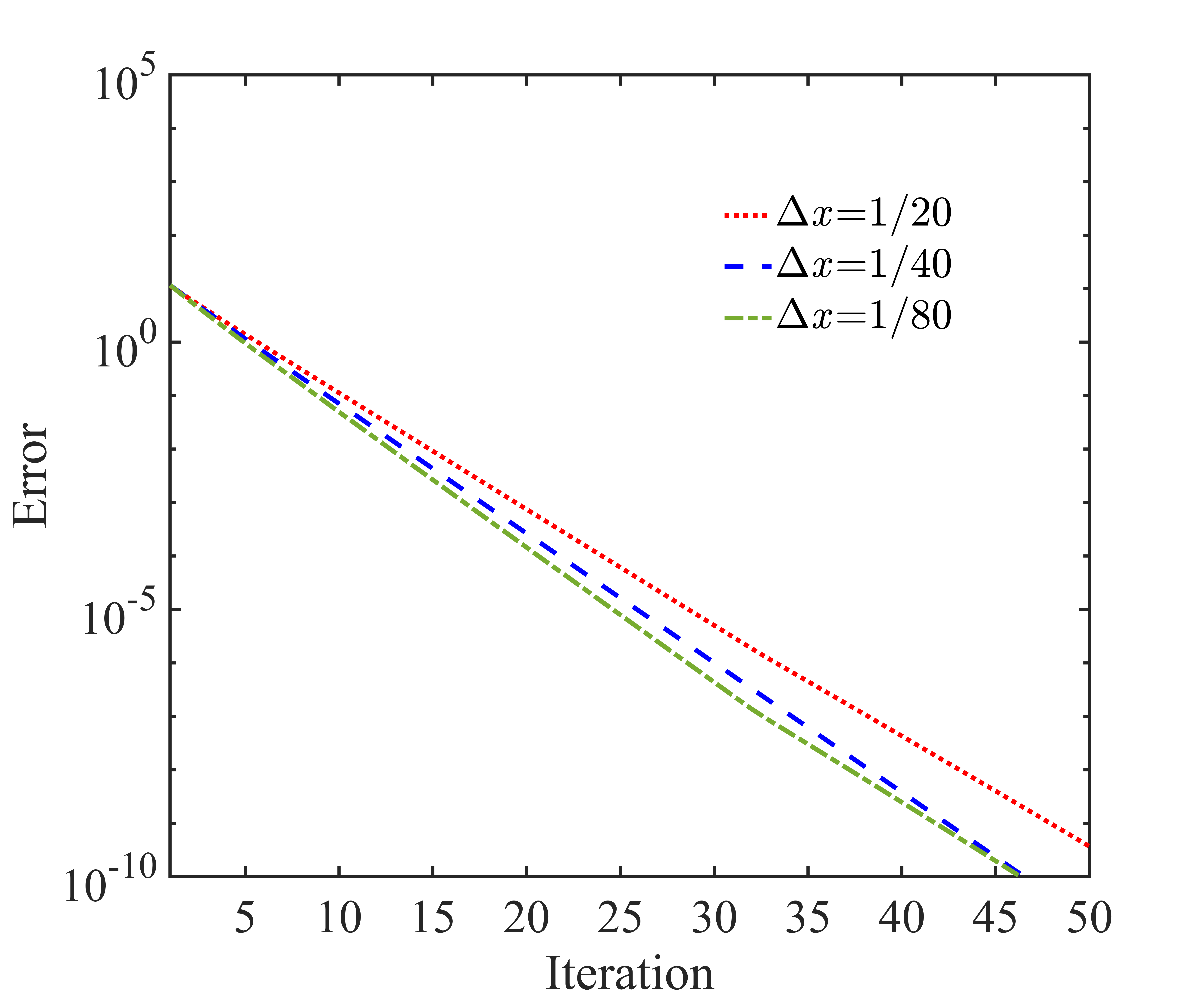}
\includegraphics[scale=0.15]{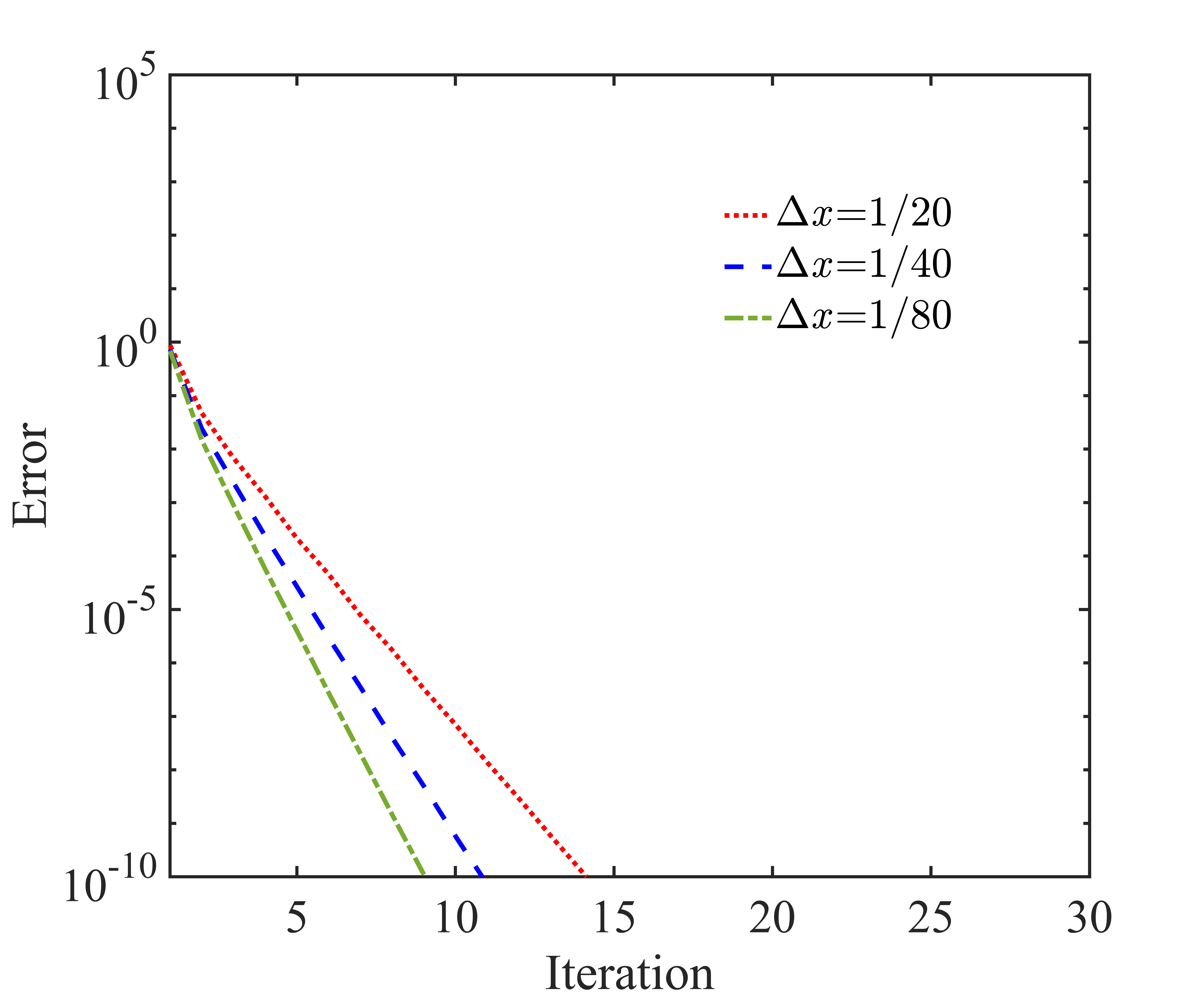} 
\includegraphics[scale=0.15]{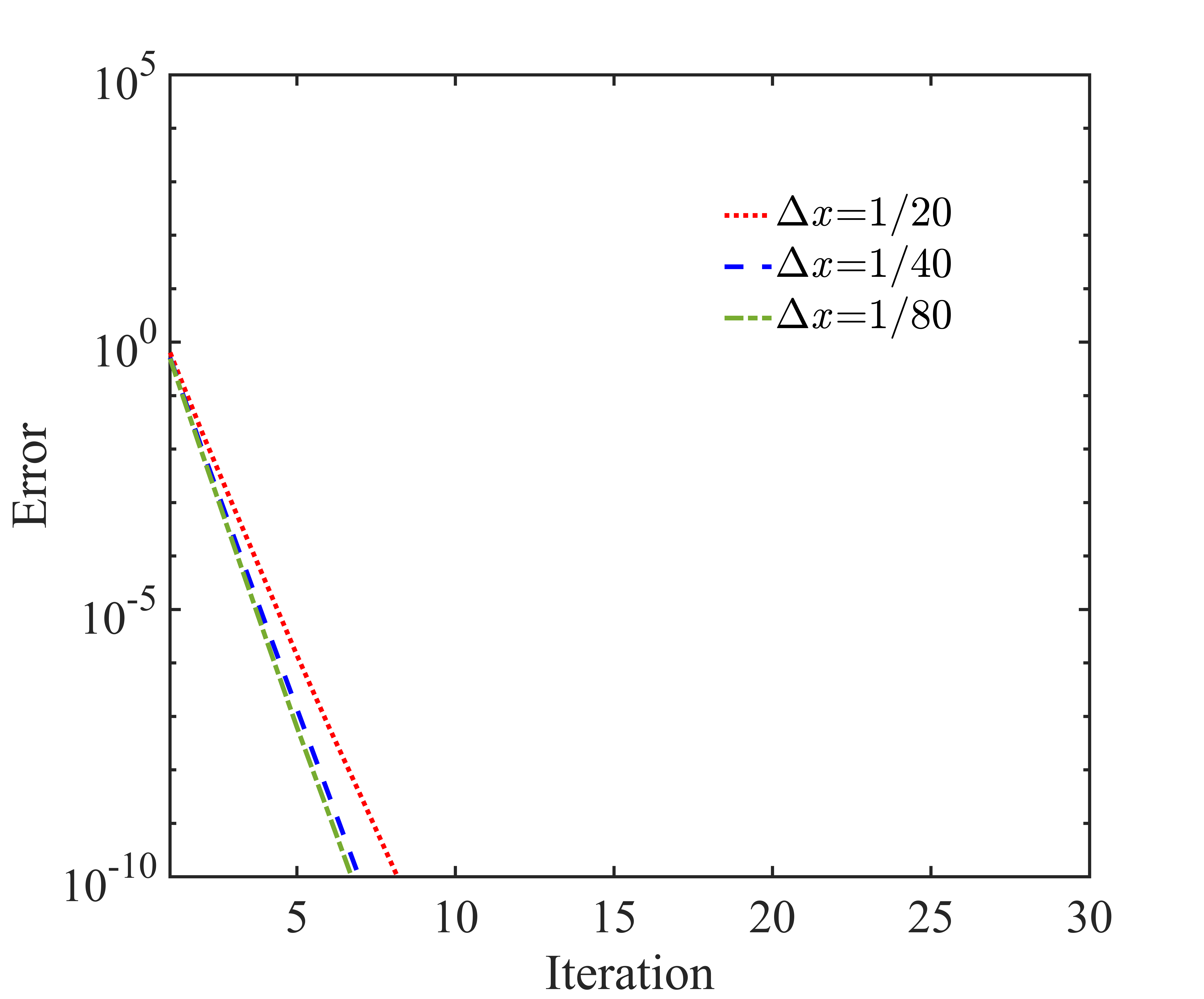} 
\end{center}
\caption{Convergence behavior of the three local transmission conditions with a given time step $\Delta t = \frac{1}{40}$ and three different mesh size $\Delta x = [\frac{1}{20}, \frac{1}{40}, \frac{1}{80}]$.  Top: $\frac{\nu_1}{\nu_2}=10$. Bottom: $\frac{\nu_1}{\nu_2}=10^3$. Left: Version I. Middle: Version II. Right: Version III. }
\label{fig:dx}
\end{figure}
Compared with the impact of the time steps, the impact of the mesh size for all three versions is relatively small, especially for the diffusion ratio $\frac{\nu_1}{\nu_2} = 10$ as shown in Figure~\ref{fig:dx} on the top. As for the ratio $\frac{\nu_1}{\nu_2} = 10^3$, we observe in Figure~\ref{fig:dx} at the bottom that the performance of all three versions is slightly improved for small mesh size in contrast to when $\Delta t$ becomes small; and once again, the convergence of Version III is more stable among all tested cases as shown in Figure~\ref{fig:dx} on the right.

\subsection{Application to thermal protection systems simulation}

To generalize our studies to practical applications, we now provide a numerical investigation of the thermal protection structure presented in Figure~\ref{fig:illustration} in a one-dimensional framework. Based on the three-layer structure of the materials, we consider a natural asymmetric decomposition with three subdomains,
\[\Omega_{1}=(0,\frac{1}{5}),
\quad 
\Omega_{2}=(\frac{1}{5},\frac{2}{5}), 
\quad
\Omega_{3}=(\frac{2}{5},1),\]
with $\Omega_{1}$ the metallic skin, $\Omega_{2}$ the strain isolation pad, and $\Omega_{3}$ the thermal insulation material. In order to imitate differences in the heat diffusion coefficient between different materials, the heat diffusion coefficients of these three subdomains are set to $1$, ${10}^{-2}$, and ${10}^{-3}$, respectively. In practice, the external temperature of the thermal insulation materials is high. Hence, to account for this, we take the Dirichlet boundary conditions $g_3=50$ at $x=1$ in $\Omega_{3}$ and $g_1=0$ at $x=0$ in $\Omega_{1}$. We set the mesh size $\Delta x = 1/100$, the time step $\Delta t=1/40$ and keep the same initial condition $u_0 = 20$.

The solution of the heat distribution is illustrated in Figure~\ref{fig:xt} on the left. 
\begin{figure}
\begin{center}
\includegraphics[scale=0.4]{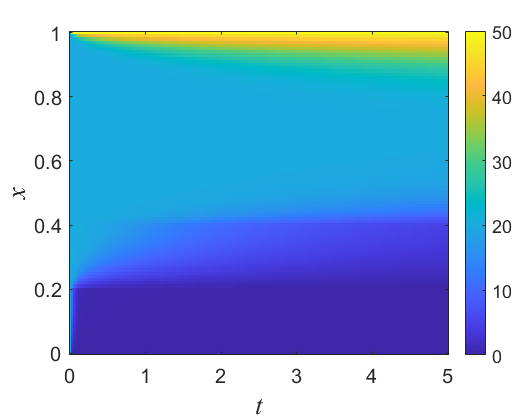} 
\includegraphics[scale=0.12]{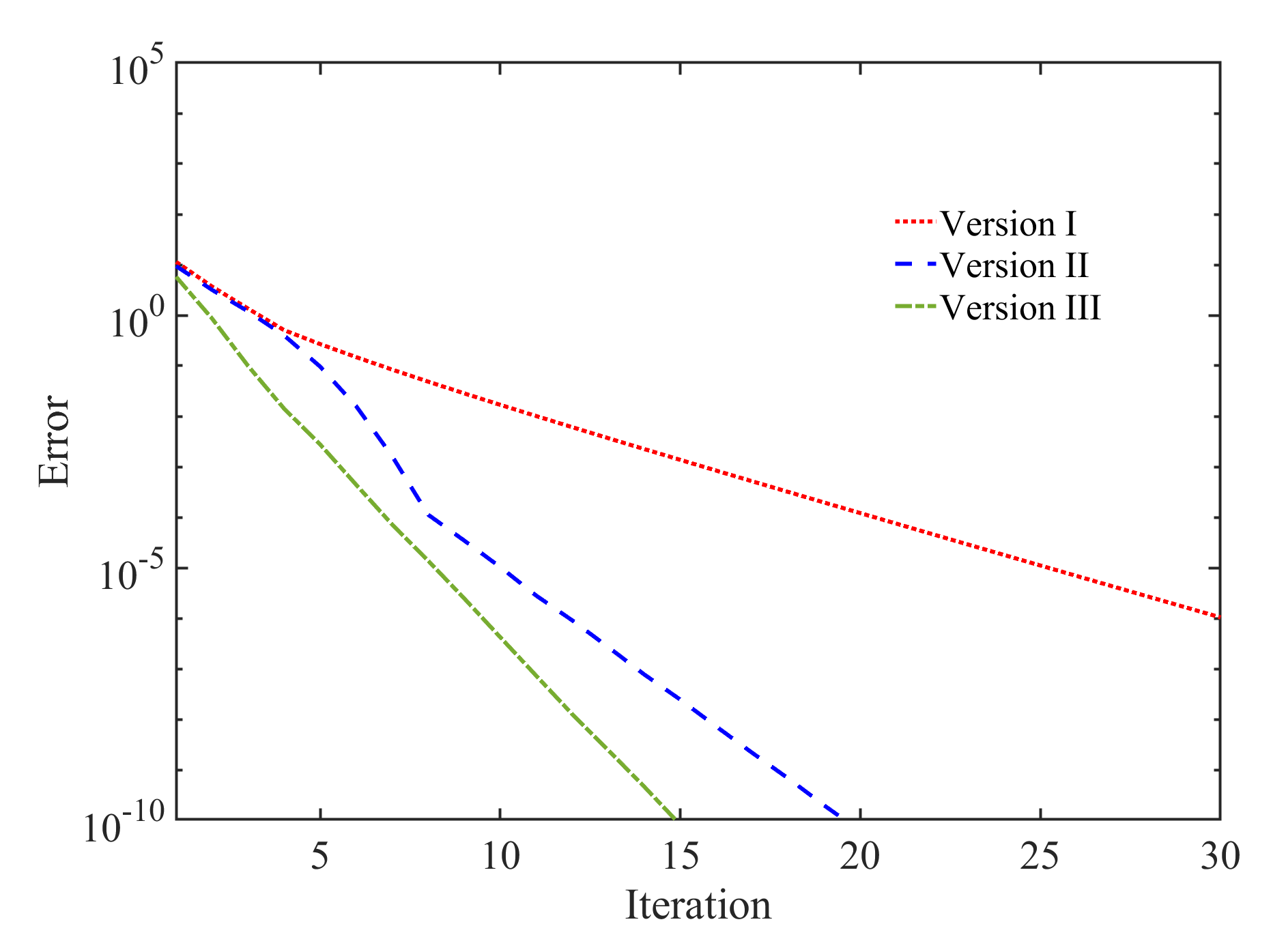}
\end{center}
\caption{Solution of the heat distribution within a thermal protection structure (Left) and convergence behavior of the three local transmission conditions with three asymmetric subdomains (Right). }
\label{fig:xt}
\end{figure}
Compared to the behavior in $\Omega_{2}$ and $\Omega_{3}$, we observe that the heat diffuses quite fast in $\Omega_{1}$ and goes rapidly to 0. However, since the heat diffusion coefficient is rather small in $\Omega_3$, it well prevents the high temperature at $x = 1$ from passing through the thermal insulation material. Furthermore, the convergence behavior of the three local transmission conditions is also presented in Figure~\ref{fig:xt} on the right. In this case with asymmetric subdomains, we observe that the convergence behavior of Versions II and III are much better than that of Version I, and Version III is the best among them. This is consistent with our previous numerical experiments, and shows that our analytical results for the two-subdomain case can provide appropriate local transmission conditions to accelerate the simulation of more general heat transfer problems within typical thermal protection structures.

\section{Conclusion}\label{sec:5}
We analyzed at the continuous level the optimized Schwarz method applied to heat transfer problems with discontinuous diffusion coefficients. We considered two nonoverlapping subdomains and optimized the transmission conditions to accelerate the convergence of the iteration. To obtain good local approximations of the transmission parameters, three local transmission parameters were studied. By solving the min-max problem associated with each transmission condition, we obtained analytical formulas for the optimized transmission parameters. These analyses can also be extended to higher dimension by using Fourier techniques, following techniques for the constant coefficient case in~\cite{Bennequin2016}. Numerical examples demonstrated that the optimized transmission conditions with an appropriate scaling are very effective and stable, and provide better convergence when the diffusion coefficient has a large discontinuity. However, the performance of all three local transmission conditions becomes rather similar when the discontinuity becomes small. In addition, we also observe in our numerical experiments that both the mesh size and the time step can influence the convergence, especially when the transmission parameters are not well scaled with respect to the diffusion coefficients. To better understand the dependency of the convergence on the mesh size and the time step, one needs to analyze the optimized Schwarz method of the discrete level in the time and space directions for such heat transfer problems. From a practical viewpoint, we showed that Version III can be used to obtain effective and robust transmission conditions to solve heat transfer problems with heterogeneous diffusion coefficients. Moreover, the numerical experiment with asymmetric decomposition and multiple subdomains also reveals the potential of the present method for realistic thermal protection structures.

\bibliographystyle{plain}
\bibliography{ref}
\end{document}